\documentclass{amsart}
\newcommand{\RN}[1]{%
  \textup{\uppercase\expandafter{\romannumeral#1}}%
}
\usepackage{amsmath}
\usepackage{amssymb}
\usepackage{enumerate}
\usepackage[hidelinks]{hyperref}

\usepackage{mathtools}
\usepackage[capitalise]{cleveref}
\numberwithin{equation}{section}

\crefname{equation}{}{}
\newtheorem{theorem}{Theorem}[section]
\newtheorem{lemma}[theorem]{Lemma}
\newtheorem{proposition}[theorem]{Proposition}

\theoremstyle{definition}
\newtheorem{definition}[theorem]{Definition}
\newtheorem{example}[theorem]{Example}

\theoremstyle{remark}
\newtheorem{remark}[theorem]{Remark}

\usepackage{tikz}
\usepackage{tikz-cd}

\usepackage{enumitem}

\begin{document}

\title{On the $K$-theoretic Hall algebra of a surface}

\author{Yu Zhao}
\address{MIT Department of Mathematics}
\email{zy199402@mit.edu}
\thanks{ }

\date{}

\dedicatory{}


\begin{abstract}
In this paper, we define the $K$-theoretic Hall algebra for
$0$-dimensional coherent sheaves on a smooth projective surface, prove
that the algebra is associative and construct a homomorphism to a
shuffle algebra analogous to Negut \cite{negut2017shuffle}).
\end{abstract}

\maketitle

\section{Introduction}

\subsection{Motivation}

Let $S$ be a smooth surface over an algebraically closed field $k$. We
consider the stack $Coh_{n}$ of length $n$ coherent sheaves on
$S$, which can be represented by a quotient stack:
$$Coh_{n}=[Quot_{n}^{\circ}/GL_{n}]$$
where $Quot_{n}^{\circ}$ is studied in \cref{sec:quot}. Therefore, the Grothendick group of $Coh_{n}$ can be
represented by
$$K(Coh_{n})=K^{GL_{n}}(Quot_{n}^{\circ})$$
and we denote $K(Coh)$ the abelian group

$$K(Coh)=\bigoplus_{n=0}^{\infty}K(Coh_{n}).$$

Based on their construction for $S=\mathbb{A}^{2}$, Schiffmann and Vasserot (\cite{schiffmann10:hall_langl}) expect there
is an algebra structure on $K(Coh)$, which is called the K-theoretic
Hall algebra of $S$. The general principle for constructing the $K$-theoretic Hall algebra
is to consider the stack $Corr_{n,m}$ of short exact sequences
$$0\to \mathcal{E}_{n}\to \mathcal{E}_{n+m}\to \mathcal{E}_{m}\to 0$$
where $\mathcal{E}_{n}\in Coh_{n}$, $\mathcal{E}_{m}\in Coh_{m}$ and
$\mathcal{E}_{n+m}\in Coh_{n+m}$ for any two non-negative integers $n,m$. There is a natural diagram:

\begin{equation*}
  \begin{tikzcd}
    & \{0\to \mathcal{E}_{n}\to \mathcal{E}_{n+m}\to \mathcal{E}_{m}\to 0\}\ar{rd} \ar{ld} & \\
    (\mathcal{E}_{n},\mathcal{E}_{m}) & & \mathcal{E}_{n+m}
  \end{tikzcd}
\end{equation*}

which induces morphisms:

\begin{equation*}
  \begin{tikzcd}
    & Corr_{n,m} \ar{rd}{p} \ar{ld}{q} & \\
    Coh_{n}\times Coh_{m} & & Coh_{n+m}
  \end{tikzcd}
\end{equation*}

where $q$ is a proper map. One expects that there is an appropriate
definition of pull back map $p^{!}:K(Coh_{n}\times Coh_{m})\to
K(Corr_{n,m})$ and push forward $q_{*}:K(Corr_{n,m})\to K(Coh_{n+m})$
such that $*^{K(Coh)}=q_{*}\circ p^{!}$ induces an associative algebra structure of $K(Coh)$.

\subsection{Description of Our Results}

In this paper, we realize the aforementioned expectation by
representing $Corr_{n,m}$ as a quotient stack
$$Corr_{n,m}=[Flag_{n,m}^{\circ}/P_{m,n}]$$
in \cref{sec:quot}. We consider a resolution of universal quotients
$$0\to \mathcal{V}_{m} \to \mathcal{W}_{m} \to \mathcal{O}^{m} \to
\mathcal{E}_{m} \to 0$$
over $Quot_{n}^{\circ}\times S$ and define two vector bundles $V_{n,m}$ and $W_{n,m}$ over
$Quot_{n}^{\circ}\times Quot_{m}^{\circ}$. We observe there is a
Cartesian diagram \cref{eq:2.4}:
 \begin{equation}
   \begin{tikzcd}
    Flag_{n,m}^{\circ } \ar{r}{t_{n,m}} \ar{d} & W_{n,m} \ar{d}{\psi_{n,m}} \\
    Quot_{n}^{\circ}\times Quot_{m}^{\circ } \ar{r}{i_{V}} & V_{n,m}    
  \end{tikzcd}
\end{equation}
where $\psi_{n,m}$ is a locally complete intersection morphism. We use
$\psi^{!}_{n,m}$ to define the refined pull-back $p^{!}$ and prove
that it does not depend on the choice of resolutions. Hence we give
the appropriate definition of $K$-theoretic Hall algebra in \cref{def:hall}. This generalizes the work of Schiffmann and Vasserot
\cite{schiffmann10:hall_langl} for $S=\mathbb{A}^{2}$ and Minets
\cite{minets18:cohom_hall_higgs} for $S=T^{*}C$ where $C$ is a smooth
projective curve. Moreover, based on the techniques
of refined Gysin maps between two vector bundles which we develop in \cref{sec:vec}, we prove that the $K$-theoretic Hall algebra is associative.

\begin{theorem}[\cref{thm:ass}]
  The K-theoretic Hall algebra $(K(Coh),*^{K(Coh)})$ is associative.
\end{theorem}

Another question we are considering in the paper is the relation
between the $K$-theoretic Hall algebra and the shuffle algebra. The shuffle
algebra is considered by Schiffmann and Vasserot \cite{schiffmann17:hall} for
quivers and Negut \cite{negut2017shuffle} for surfaces.

The idea is to consider $T_{n}\subset GL_{n}$ which is the maximal
torus consisting of diagonal matrices. The fixed locus is shown to be
$(Quot_{n}^{\circ })^{T_{n}}=(Quot_{1}^{\circ })^{n}=S^{n}$ in
\cref{lem:2.4}. By \cref{thm:1.3} and the Thomason localization \cref{thm:1.4}, we have
$$K^{G_{n}}(Quot_{n}^{\circ })=(K^{T_{n}}(Quot_{n}^{\circ
}))^{\sigma_{n}}$$
and
$$K^{T_{n}}(Quot_{n}^{\circ})_{loc}=K^{T_{n}}(S^{n})_{loc}$$

where $\sigma_{n}$ is the permutation group of order $n$, and "loc" denotes localization over the fraction field of $Rep(T_{n})$.

Let
$$Sh=\bigoplus_{n=0}^{\infty}K^{T_{n}}(S^{n})_{loc}^{\sigma_{n}}.$$

Sh is endowed with a shuffle algebra structure, as in \cite{negut2017shuffle}, and we show that the localization theorem induces an algebra homomorphism \cref{eq:tau} from K(Coh) to a renormalized version (see \cref{def:full surface big}) of the shuffle algebra Sh.

\begin{theorem}[Theorem \ref{thm:3.4}]
  $\tau$ is an algebra homomorphism between $(K(Coh),*^{K(Coh)})$ and $(Sh,*^{Sh})$.
\end{theorem}

\subsection{Structure of the Paper}
In Section \ref{sec:K-theory}, we review some basic facts about
equivariant K-theory, like refined Gysin maps, induction and some
applications to the vector bundles.

In Section \ref{sec:quot}, we introduce the Quot schemes and Flag
schemes, and resolution of the universal quotients over those
schemes. We also discuss some properties of refined Gysin maps
between Quot schemes and Flag schemes, which will be used in the
discussion of K-theoretic Hall algebra.

In Section \ref{sec:hall}, we define the K-theoretic Hall algebra of
a surface, and prove this algebra is associative.

In Section \ref{sec:shuffle}, we redefine the shuffle algebra $Sh$ of
\cite{negut2017shuffle} and construct a homomorphism from $K(Coh)$ to $Sh$.

\subsection{Relations to Other Work}

In the case $S=\mathbb{A}^{2}$, Schiffmann and Vasserot (\cite{schiffmann10:hall_langl})
studied an equivariant version of $K$-theoretic Hall algebra. In
the case $S=T^{*}C$ where $C$ is a smooth curve, Minets (\cite{minets18:cohom_hall_higgs}) studied
an analogous moduli stack, namely the moduli stack of Higgs sheaves.

Instead of studying K-theory, one could enquire about other cohomology theories, and we expect many of the results in the present paper to carry through (assuming the existence of an equivariant localization theorem). For example, in \cite{kapranov2019cohomological}, Kapranov-Vasserot independently introduced a construction philosophically similar to ours, and constructed an algebra structure on the Borel-Moore homology groups of the stacks of coherent sheaves of arbitrary dimension on S.

\subsection{Acknowledgment}

This paper is dedicated to my advisor, Andrei Negut. It is hard to
imagine any academic breakthrough of myself without his tremendous
unselfish help. I would also like to acknowledge Shuai Wang, who first
introduced this topic to me.

When I was writing this paper, I received lots of useful suggestions
form Davesh Maulik, and after posting the first paper to arxiv, I also
received many useful feedbacks from Alexandre Minets, Francesco Sala
and Mikhail Kapranov. I would like to
acknowledge them for their useful discussions.
\section{Equivariant K-theory}\label{sec:K-theory}

In this section, we recall some basic facts about equivariant
K-theory from \cite{chriss2009representation} and
\cite{anderson2015operational}. We work in the category of separated
schemes of finite type over an algebraically closed field $k$,
equipped with an action of a reductive group $G$. All morphisms are equivariant with respect to the action of $G$.

\subsection{Grothendick groups}
The Grothendick group of equivariant coherent sheaves $K^{G}(X)$ is
generated by classes $[\mathcal{F}]$ for each $G$-equivariant coherent sheaf $\mathcal{F}$ on
$X$, subject to the relation $[\mathcal{F}]=[\mathcal{G}]+[\mathcal{H}]$ for any exact sequence of $G$-equivariant coherent sheaves
$$0\to \mathcal{G}\to \mathcal{F}\to \mathcal{H}\to 0$$

\begin{example}
  Let $X=\mathit{Spec}(k)$. Then $K^{G}(X)=Rep(G)$, the representation
  ring of $G$. Moreover, for any scheme $X$, $K^{G}(X)$ is a module over $Rep(G)$.
\end{example}

\subsection{Refined Gysin Maps}
\label{sec:1.1}

\begin{definition}
  A morphism $f:X\to Y$ is called a locally complete intersection
  morphism (l.c.i. morphism for short) if $f$ is the composition of a regular embedding and a smooth
 morphism.
\end{definition}

\begin{example}
  In this paper, we will not distinguish locally free sheaves and vector
bundles as their total spaces. For any vector bundle $M$ over $X$, let
$pr_{M}:M\to X$ be the projection and $i_{M}:X\to M$ be the zero
section.

Given $\phi:V\to W$ a linear morphism of two vector bundles
over $X$. Let $Y=\phi^{-1}(i_{W}(X))$ and we consider the following cartesian diagram:

\begin{equation*}
  \begin{tikzcd}
    Y \ar{r} \ar{d} & V \ar{d}{\rho} \\
    V \ar{r}{(id,0)} \ar{d}{pr_{V}} & V\times _{X}W \ar{d}{pr_{W}} \\
    X \ar{r}{i_{W}}  & W
  \end{tikzcd}
\end{equation*}
where $\rho=(id,\phi)$. $pr_{W}$ is flat and $\rho$ is a regular
embedding, so $\phi=pr_{W}\circ\rho$ is l.c.i..
\end{example}

\begin{definition}
  Given a Cartesian diagram
\begin{equation}
  \begin{tikzcd}
    X'\ar{r}{g'} \ar{d}{f'}& X \ar{d}{f} \\
    Y'\ar{r}{g}            & Y          \\
  \end{tikzcd}
\end{equation}
where $f$ is a l.c.i. morphism, the
\textbf{refined Gysin map} $f^{!}:K(Y')\to K(X')$ is defined by
$$f^{!}([\mathcal{F}])=\sum_{i=0}^{\infty}(-1)^{i}[Tor_{i}^{\mathcal{O}_{Y}}(\mathcal{O}_{X},\mathcal{F})].$$

$f^{!}$ is well defined because $f$ has finite Tor-dimension,
i.e. for any $\mathcal{F}\in Coh(Y)$,
$Tor_{i}^{\mathcal{O}_{Y}}(\mathcal{O}_{X},\mathcal{F})=0$ for $i\gg
0$ if we regard $\mathcal{O}_{X}$ as a $\mathcal{O}_{Y}$-module.
This definition still holds if $\mathcal{F}$ has an equivariant structure.
\end{definition}

The refined Gysin map has the following properties:

\begin{lemma}[Lemma 3.1 of \cite{anderson2015operational}]
  \label{lem:1.1}
  Consider the following Cartesian diagrams \cref{cd:1}:
  \begin{equation}
    \label{cd:1}
    \begin{tikzcd}
      X''\ar{r}{h'}  \ar{d}  & X' \ar{r}{g'} \ar{d}{f'}  & X \ar{d}{f} \\
      Y''\ar{r}{h} & Y' \ar{r}{g}  & Y,
    \end{tikzcd}
  \end{equation}
where $h$ is proper and $f$ is a l.c.i. morphism. Then
$$f^{!}h_{*}=h_{*}'f^{!}:K(Y'')\to K(X').$$
\end{lemma}

\begin{lemma}[Lemma 3.2 of \cite{anderson2015operational}]
  \label{lem:kcom}
  Consider following Cartesian diagrams \cref{cd:2}:
  \begin{equation}
    \label{cd:2}
    \begin{tikzcd}
      X'' \ar{r}\ar{d} & Y'' \ar{r}\ar{d} & Z'' \ar{d}{h} \\
      X'  \ar{r}\ar{d} & Y'  \ar{r}\ar{d} & Z'            \\
      X   \ar{r}{f}    & Y.
    \end{tikzcd}
  \end{equation}
  such that $h$ and $f$ are l.c.i. morphisms. Then
  $$f^{!}h^{!}=h^{!}f^{!}:K(Y')\to K(X'').$$
\end{lemma}

\begin{lemma}
  \label{lem:flat}
  Consider the following Cartesian diagrams \cref{cd:30}:
  \begin{equation}
    \label{cd:30}
    \begin{tikzcd}
      X'' \ar{r} \ar{d}{f''} & X' \ar{r} \ar{d}{f'} & X \ar{d}{f} \\
      Y'' \ar{r} & Y' \ar{r}{s} & Y, \\
    \end{tikzcd}
  \end{equation}
  where $f$ and $f'$ are l.c.i morphisms. If one of the $f$ and $s$ is
  flat, 
  $$f^{!}=f'^{!}:K(Y'')\to K(X'')$$ 
  and if $f$ is flat, $f^{!}=f''^{*}$.
\end{lemma}
\begin{proof}
  Recall the Tor spectral sequence
     \cite[\href{https://stacks.math.columbia.edu/tag/061Y}{Tag 061Y}]{stacks-project}

  \begin{equation}
    \label{eq:tor}
  Tor_{p+q}^{\mathcal{O}_{Y}}(\mathcal{F},\mathcal{O}_{X})\Rightarrow Tor_{p}^{\mathcal{O}_{Y'}}(\mathcal{F},Tor_{q}^{\mathcal{O}_{Y}}(\mathcal{O}_{Y'},\mathcal{O}_{X})).  
  \end{equation}
  Since $f$ or $s$ is flat, we have
  $Tor_{q}^{\mathcal{O}_{Y}}(\mathcal{O}_{Y'},\mathcal{O}_{X})=0$ for
  $q>0$ and
  $Tor_{0}^{\mathcal{O}_{Y}}(\mathcal{O}_{Y'},\mathcal{O}_{X})=\mathcal{O}_{X'}$. Hence for any $\mathcal{F}\in K(Y'')$
  \begin{align*}
    \sum_{i=0}^{\infty }[(-1)^{i}Tor_{i}^{\mathcal{O}_{Y}}(\mathcal{F},\mathcal{O}_{X})] & =\sum_{p=0}^{\infty }\sum_{q=0}^{\infty }[(-1)^{p+q}Tor_{p}^{\mathcal{O}_{Y'}}(\mathcal{F},Tor_{q}^{\mathcal{O}_{Y}}(\mathcal{O}_{Y'},\mathcal{O}_{X}))] \\
    & =\sum_{i=0}^{\infty }[(-1)^{i}Tor_{i}^{\mathcal{O}_{Y'}}(\mathcal{F},\mathcal{O}_{X'})].
  \end{align*}

\end{proof}
Another corollary of the Tor spectral sequence \cref{eq:tor} is the
following lemma:
\begin{lemma}
  \label{lem:k-ass}
  Consider the following Cartesian diagram:
  \begin{equation*}
    \begin{tikzcd}
      X' \ar{r} \ar{d} & X \ar{d}{f} \\
      Y' \ar{r} \ar{d} & Y \ar{d}{g} \\
      Z' \ar{r} & Z 
    \end{tikzcd}
  \end{equation*}
  where $f$, $g$ and $h=g\circ f$ are l.c.i morphisms. Then
  $h^{!}=f^{!}\circ g^{!}:K(Z')\to K(X')$.
\end{lemma}

\subsection{Refined Gysin Maps Between Vector Bundles}\label{sec:vec}
 Let
  \begin{equation*}
    \begin{tikzcd}
      0 \ar{r} & W_{1} \ar{r}{g'} \ar{d}{\phi_{1}} & W_{2} \ar{r}{f'} \ar{d}{\phi_{2}} & W_{3} \ar{r} \ar{d}{\phi_{3}} & 0 \\
      0 \ar{r} & V_{1} \ar{r}{g}  & V_{2} \ar{r}{f} & V_{3} \ar{r} & 0  
    \end{tikzcd}
  \end{equation*}
  be a commutative diagram of vector bundles over $X$ where all the
  rows are exact. Let $X_{3}=\phi_{3}^{-1}(i_{V_{3}}(X))$,
  $X_{2}=\phi_{2}^{-1}(i_{V_{2}}(X))$ and $Y=f'^{-1}(X_{3})$. Then $Y$ is an
  affine bundle over $X_{3}$. The following commutative diagram:
  \begin{equation*}
    \begin{tikzcd}
      Y\ar{r}{\phi_{2}} \ar{d}{pr_{W_{2}}} & V_{2} \ar{d}{f} \\
      X \ar{r}{i_{V_{3}}} & V_{3}
    \end{tikzcd}
  \end{equation*}
induces a l.c.i morphism $\psi:Y\to V_{1}\times_{X}
X_{3}$ and the following cartesian diagram:  
\begin{equation}
    \label{di:final1}
    \begin{tikzcd}
      X_{2}\ar{r} \ar{d} & Y \ar{d}{\psi} \\
      X_{3} \ar{r}{i_{V_{1}}\times id} & V_{1}\times _{X}X_{3}   
    \end{tikzcd}
  \end{equation}

\begin{lemma}
  \label{lem:kass}
  $\psi^{!}\circ \phi_{3}^{!}=\phi_{2}^{!}:K(X)\to K(X_{2})$.
\end{lemma}
\begin{proof}
  Consider the following two Cartesian diagrams:
  
  \begin{minipage}[t]{0.5\linewidth}
  \begin{equation*}
    \begin{tikzcd}
      Y \ar{r} \ar{d} & X_{3} \ar{r}\ar{d}&  X \ar{d} \\
      W_{2} \ar{r}{f'} & W_{3} \ar{r}{\phi_{3}} & V_{3}
    \end{tikzcd}
  \end{equation*}
  \end{minipage}
  \begin{minipage}[t]{0.5\linewidth}
    \begin{equation*}
      \begin{tikzcd}
        X_{2} \ar{r} \ar{d} & X \ar{d}{i_{V_{1}}} & \\
        Y \ar{r} \ar{d} & V_{1} \ar{r}{pr_{V_{1}}} \ar{d}{g} & X \ar{d} \\
        W_{2} \ar{r}{\phi_{2}} & V_{2} \ar{r}{f} & V_{3} 
      \end{tikzcd}
    \end{equation*}
  \end{minipage}
 where $\phi_{3} \circ f'= f\circ \phi_{2}$. Then
 $\phi_{2}^{!}:K(X)\to K(X_{2})$ satisfies
  \begin{align*}
    \phi_{2}^{!} & =\phi_{2}^{!}\circ i_{V_{1}}^{!}\circ pr_{V_{1}}^{!} & K(X_{2}) \xleftarrow{\phi_{2}^{!}} K(X) \xleftarrow{i_{V_{1}}^{!}} K(V_{1}) \xleftarrow{pr_{V_{1}}^{!}} K(X) \\
              &= i_{V_{1}}^{!}\circ \phi_{2}^{!}\circ f^{!} &K(X_{2}) \xleftarrow{i_{V_{1}}^{!}} K(Y) \xleftarrow{\phi_{2}^{!}} K(V_{1})  \xleftarrow{pr_{V_{1}}^{!}} K(X) & \ \text{by Lemma \ref{lem:kcom}}\\
    &=i_{V_{1}}^{!}\circ f'^{!}\circ \phi_{3}^{!} & K(X_{2})\xleftarrow{i_{V_{1}}^{!}} K(Y) \xleftarrow{f'^{!}} K(X_{3}) \xleftarrow{\phi^{!}_{3}} K(X) & \ \text{by Lemma \ref{lem:k-ass}}
  \end{align*}
 On the other hand, consider the following Cartesian diagram:
  \begin{equation*}
    \begin{tikzcd}
      X_{2} \ar{r} \ar{d} & X_{3}\ar{r} \ar{d} & X \ar{d}{i_{V_{1}}} \\
      Y \ar{r}{\psi}  &  V_{1}\times _{X}X_{3} \ar{r} \ar{d}{pr_{V_{1}}|_{X_{3}}} & V_{1}  \ar{d}{pr_{V_{1}}} \\
      & X_{3} \ar{r} & X
    \end{tikzcd}
  \end{equation*}
  and thus
  \begin{align*}
    \psi^{!}&=\psi^{!}\circ i_{V_{1}}^{!}\circ pr_{V_{1}}^{!}  & K(X_{2}) \xleftarrow{\psi^{!}} K(X_{3}) \xleftarrow{i_{V_{1}}^{!}}K(V_{1}\times_{X}X_{3}) \xleftarrow{pr_{V_{1}}^{!}} K(X_{2}) \\
         &=i_{V_{1}}^{!}\circ \psi^{!}\circ pr_{V_{1}}^{!}& K(X_{3})\xleftarrow{i_{V_{1}}^{!}}K(Y)\xleftarrow{\psi^{!}}K(Y_{1}\times_{X}X_{3})\xleftarrow{pr_{V_{1}}^{!}} K(X_{2}) &\ \text{by Lemma \ref{lem:kcom}}\\
            &=i_{V_{1}}^{!}\circ f'^{!} &K(X_{3})\xleftarrow{i_{V_{1}}^{!}} K(Y) \xleftarrow{f'^{!}} K(X_{2}) & \ \text{by Lemma \ref{lem:k-ass}} & \ 
  \end{align*}
  as morphisms from $K(X_{3})$ to $K(X_{2})$. Hence $\psi^{!}\circ \phi_{3}^{!}=\phi_{2}^{!}$.
\end{proof}
\subsection{Localization Theorem}

Let $T=(\mathbb{G}_{m})^{n}$ be the maximal torus of $G$, $X^{T}$ the
fixed locus of $X$ with respect to a given $T$ action and $i_{X}$ the
closed embedding of $X^{T}$ into $X$. We recall following theorems
about the relation between $K^{G}(X)$ and $K^{T}(X^{T})$.

\begin{theorem}[Proposition 31 of \cite{merkurjev2005equivariant}]
  \label{thm:1.3}
  If the commutator group of $G$ is simply connected, then the natural restriction map
  $K^{G}(X)\to K^{T}(X)$ gives rise to an isomorphism
  $$K^{G}(X)=K^{T}(X)^{W}.$$
\end{theorem}

\begin{definition}
   The localization of $K^{T}(X)$ is defined by
  $$K^{T}(X)_{loc}=K^{T}(X)\otimes _{Rep(T)}Frac(Rep(T)),$$ where $Frac(Rep(T))$ is the
  fraction field of $Rep(T)$.
\end{definition}
\begin{theorem}[Thomason localization theorem, Theorem 2.2 of \cite{thomason1992formule}]
  \label{thm:1.4}
  The map:
  $$K^{T}(X^{T})_{loc}\xrightarrow{i_{X*}}K^{T}(X)_{loc}$$
  is an isomorphism.
  Moreover, if $i_{X}$ is a regular embedding, then
  $$K^{T}(X)_{loc}\xrightarrow{i^{*}_{X}}K^{T}(X^{T})_{loc}$$
  is also an isomorphism.
\end{theorem}

Given $E$ a rank $n$ locally free sheaf over $X$, let
$[\wedge^{\bullet}(E)]=\sum_{i=0}^{n}(-1)^{i}[\wedge^{i}E]$. One corollary of the localization theorem is that

\begin{lemma}[Lemma 5.1.1 of \cite{ciocan-fontanine07:virtual}]
  Let $X$ be a quasi-projective scheme and assumes $T$ acts on $X$
  trivially. Then for every $T$-equivariant vector bundle $E$, identifying it with its total space, satisfying $E^{T}=X$, the
  element $[\wedge^{\bullet}(E)]$ is invertible in $K^{T}(X)_{loc}$. 
\end{lemma}

The following proposition reveals the relation between pull back and push
forward maps of Grothendieck groups.
\begin{lemma}[Proposition 5.4.10 of \cite{chriss2009representation}]
  \label{lem:1.6}
  Let $i:N\xhookrightarrow{}M$ be a $G$-equivariant closed regular
  embedding. Then the conormal sheaf $T_{N}^{*}M$ is locally free and
  the composite map
  $K^{G}(N)\xrightarrow{i_{*}}K^{G}(M)\xrightarrow{i^{*}}K^{G}(N)$ is
  given by the formula $i^{*}i_{*}\mathcal{F}=[\wedge^{\bullet}T_{N}^{*}M]\otimes \mathcal{F}$, for any $\mathcal{F}\in K^{G}(N)$.
\end{lemma}

One application of Lemma \ref{lem:1.6} is that

\begin{lemma}
  \label{lem:1.8}
  Let $X$ be a quasi-projective scheme with trivial $T$ action. Let
  $W$ and $V$ be two $T$-equivariant vector bundles satisfying
  $W^{T}=X$ and $V^{T}=X$, and $f:W\to V$ a $T$-equivariant linear morphism.
  Let $Y=f^{-1}(i_{V}(X))$, with the following Cartesian diagram:
  \begin{equation*}
    \begin{tikzcd}
    Y \ar{r}{i'} \ar{d}{f'}& W \ar{d}{f} \\
    X            \ar{r}{i_{V}} & V
    \end{tikzcd}
  \end{equation*}
  where $i_{V}$ is the zero section in $V$. Then $Y^{T}=X$. Let 
  $X\xhookrightarrow{j}Y$ be the embedding of fixed points. For any
  $\mathcal{F}\in K^{T}(X)_{loc}$,
  \begin{equation}
    \label{eq:1.4}
    f^{!}\mathcal{F}=j_{*}(\frac{[\wedge^{\bullet}V^{∨}]}{[\wedge^{\bullet}W^{∨}]}\mathcal{F})
  \end{equation}
  as elements in in $K^{T}(Y)_{loc}$.
  
\end{lemma}
\begin{proof}
  The fact that $Y^{T}=X$ is trivial.
  By the Thomason localization theorem
  \ref{thm:1.4}, the following morphisms are isomorphisms:
  $$K^{T}(X)_{loc}\xrightarrow{j_{*}}K^{T}(Y)_{loc}\xrightarrow{i'_{*}} K^{T}(W)_{loc}.$$
  Let $i_{W}$ be the zero section of $W$, then 
  $$i_{W}^{*}:K^{T}(W)_{loc}\to K^{T}(X)_{loc}$$
  is also an isomorphism, by  \cref{lem:1.6} and \cref{thm:1.4}. Thus
  equality \cref{eq:1.4} is equivalent to the following equality:
  $$i_{W}^{*}i_{*}'f^{!}\mathcal{F}=i_{W}^{*}i_{*}'j_{*}(\frac{[\wedge^{\bullet}V^{∨}]}{[\wedge^{\bullet}W^{∨}]}\mathcal{F}),$$
  i.e.
  \begin{equation}
    \label{eq:1.5}
    i_{W}^{*}(\sum_{i=0}^{\infty }(-1)^{i}[Tor_{i}^{V}(\mathcal{F},\mathcal{O}_{W})])=i_{W}^{*}i_{W_{*}}(\frac{[\wedge^{\bullet}V^{∨}]}{[\wedge^{\bullet}W^{∨}]}\mathcal{F}).  
  \end{equation}
  The right hand side of \cref{eq:1.5} is $[\wedge^{\bullet}V^{∨}]\mathcal{F}$ by \cref{lem:1.6}. The left hand side is
  $$\sum_{j=0}^{\infty }\sum_{i=0}^{\infty }(-1)^{i+j}[Tor_{j}^{W}(Tor_{i}^{V}(\mathcal{F},\mathcal{O}_{W}),\mathcal{O}_{X})]=\sum_{i=0}^{\infty }(-1)^{i}[Tor_{i}^{V}(\mathcal{F},\mathcal{O}_{X})]$$
  by equation \eqref{eq:tor}, which is also $[\wedge^{\bullet}V^{∨}]\mathcal{F}$ by \cref{lem:1.6}.
\end{proof}
\subsection{Induction}

Let $P\subset G$ be a closed algebraic subgroup which acts on $X$. The
the induced space, $G\times _{P}X$ is defined to be the space of orbits of
$P$ acting freely on $G\times X$ by $h:(g,x)\to
(gh^{-1},hx)$. Formula (5.2.17) of \cite{chriss2009representation} constructs two isomorphisms $res$ and $ind_{P}^{G}$ which are the inverse of each other:
\begin{equation}
  \label{eq:ind}
  K^{P}(X)\overset{res}{\underset{ind_{P}^{G}}\rightleftarrows}
K^{G}(X\times_{G}P).
\end{equation}

Now let $P$ be a parabolic subgroup, $T\subset P$ be a maximal torus,
and $H$ be the Levi subgroup of $P$.  Let $W$ be the Weyl group of
$G$, and $W_{H}$ be the Weyl group of $H$. Let $\mathfrak{g}$ and
$\mathfrak{p}$ be the Lie algebras of $G$ and $P$ and $T$ acts on
$\mathfrak{g}$ and $\mathfrak{p}$ by adjoint representations.

\begin{lemma}
  \label{lem:1.9}
  Let $Y=X\times _{P}G$, then:
  \begin{enumerate}
  \item\label{1.9.1} $Y^{T}=X^{T}\times _{W_{H}}W$.
  \item Let $s:Y^{T}\to X^{T}$ to be a projection associated to a
    choice of representatives for $W_{H}$ in $W$, $j:X\to Y$ to be the natural
    inclusion $x\to (x,e)$, where $e$ is the unit of $G$. If $X^{T}$
    is connected, we have the following commutative diagram:
    \begin{equation}
    \label{eq:1.7}
    \begin{tikzcd}
       K^{T}(Y^{T}) \ar{r}{i_{Y*}} & K^{T}(Y) \ar{d}{j^{*}} \\
      K^{T}(X^{T}) \ar{r}{\tilde{i}} \ar{u}{s^{*}}& K^{T}(X)
    \end{tikzcd}
  \end{equation}
  where
  $\tilde{i}(\mathcal{F})=([\wedge^{\bullet}(\mathfrak{g}/\mathfrak{p})^{*}])\mathcal{F}$.
  \item\label{1.9.3} We have the commutative diagram:
  \begin{equation}
    \label{eq:res}
    \begin{tikzcd}
      K^{G}(Y)\ar{r} \ar{d}{res} & K^{T}(Y) \ar{d}{j^{*}} \\
      K^{P}(X)\ar{r} & K^{T}(X)
    \end{tikzcd}
  \end{equation}
  \end{enumerate}
\end{lemma}
\begin{proof}
  \eqref{1.9.1} was proven in Proposition A.17 of
  \cite{minets18:cohom_hall_higgs} and \eqref{1.9.3} is obvious from
  definition.
  Let $j^{T}$ be the embedding of $X^{T}$ in $Y^{T}$. We have the
  Cartesian square:
   \begin{equation*}
    \begin{tikzcd}
      X^{T} \ar{r}{i_{X}}\ar{d}{j_{T}} & X \ar{d}{j} \\
      Y^{T} \ar{r}{i_{Y}}       & Y.         
    \end{tikzcd}
  \end{equation*}
  and thus $i_{Y*}\circ j_{T*}=j_{*}\circ i_{X*}:K(X^{T})\to K(Y)$. Moreover, we
  have
  $$Y^{T}=\bigcup_{i=1}^{k}Y_{i},$$
  where all $Y_{i}$ are different connected components of $Y^{T}$ and
  $Y_{1}=X^{T}$, the only component which has non-empty intersection
  with $X$. Thus $s^{*}(\mathcal{F})=j_{T*}\mathcal{F}+\mathcal{F}'$, where $\mathcal{F}'$ is supported in
  $\bigcup_{i=2}^{k}Y_{i}$. So $j^{*}\circ i_{Y*}(\mathcal{F}')=0$ and we
  have
  \begin{align*}
    j^{*}\circ i_{Y*}\circ s^{*}(\mathcal{F}) & =j^{*}\circ i_{Y*}\circ j_{T*}(\mathcal{F}) \\
                                              & = j^{*}\circ j_{*} \circ i_{X*}(\mathcal{F}) \\
    & =[\wedge^{\bullet}T_{X}Y^{*}]i_{X*}(\mathcal{F})
  \end{align*}
 by Lemma \ref{lem:1.6}. Now we prove $T_{X}Y=\mathfrak{g}/\mathfrak{p}$. Notice that $Y$ is a fiber bundle over $G/P$ with fiber $X$:
  \begin{equation*}
    \begin{tikzcd}
      X\ar{r}{j} \ar{d}{\pi} & Y \ar{d} \\
      Spec(k) \ar{r}{e} & G/P,
    \end{tikzcd}
  \end{equation*}
  where the bottom line corresponds to the unit element. Hence $T_{X}Y=\pi^{*}(\mathfrak{g}/\mathfrak{p})=\mathfrak{g}/\mathfrak{p}$.
 \end{proof}
\subsection{K-theory of Artin stacks}

The concept of Grothendieck group can also be extended to the case of
Artin stacks. \cite{toen99:k} is a great reference for the discussion. A special case of Artin stacks is group
quotient, i.e. $Y=[X/H]$, where $X$ is a scheme and $H$ is an
algebraic group with group action on $X$, then we have:

\begin{lemma}[Lemma 2.11 of \cite{toen99:k}]
  $$K(X/H)=K^{H}(X).$$
\end{lemma}

\section{Quot Schemes and Flag Schemes of a Surface}\label{sec:quot}

In this paper, we will work over an algebraically closed field $k$ of
characteristic 0. Let $S$ be a smooth projective surface and
$\mathcal{O}(1)$ an ample line bundle over $S$.

\subsection{Quot Schemes and Flag Schemes}
\begin{definition}
  Given $d$ a non-negative integer, Grothendieck's Quot scheme
  $Quot_{d}$ is defined to be the moduli scheme of quotients of coherent sheaves
$$\{\phi_{d}:k^{d}\otimes \mathcal{O}_{S}\twoheadrightarrow
\mathcal{E}_{d}\},$$ where $\mathcal{E}_{d}$ has dimension $0$ and
$h^{0}(\mathcal{E}_{d})=d$.

There is an open subscheme  $Quot_{d}^{\circ }\subset Quot_{d}$ which
consists of quotients such that $H^{0}(\phi):k^{d}\to
H^{0}(\mathcal{E}_{d})$ is an isomorphism.

Over $Quot_{d}$ and $Quot_{d}^{\circ}$, there is a universal quotients
coherent sheaves $$\phi_{d}:k^{d}\otimes \mathcal{O}_{S} \twoheadrightarrow
\mathcal{E}_{d}.$$
Its kernel, denoted by $\mathcal{I}_{d}$. is
defined as a \textbf{kernel
sheaf}.
\end{definition}

\begin{example}
  $Quot_{1}^{\circ }=S$. Let $\Delta:S\to S\times S$ be the diagonal map, then
  $\mathcal{E}_{1}\in Coh(Quot_{1}^{\circ }\times S)$ is
  $\mathcal{O}_{\Delta}$, and $\mathcal{I}_{1}$ is
  $\mathcal{I}_{\Delta}$, the ideal sheaf of diagonal. 
\end{example}

\begin{remark}
  There is a principle of abusing notations in this paper. For any
  scheme $X$ and $f:X\to Quot_{n}^{\circ }$, let
  \begin{center}
    \begin{tabular}{ccc}
      $f^{*}\mathcal{E}_{d}\in Coh(X\times S)$ &  denoted by & $\mathcal{E}_{d}$ \\
      $f^{*}\mathcal{I}_{d}\in Coh(X\times S)$ & denoted by & $\mathcal{I}_{d}$
    \end{tabular}
  \end{center}
\end{remark}

\begin{lemma}
  \label{lem:res}
  There exists a short exact sequence
  \begin{equation}
    \label{eq:2.1}
     0\to \mathcal{V}_{d}\xrightarrow{v_{d}}\mathcal{W}_{d}\xrightarrow{w_{d}}\mathcal{I}_{d}\to 0,
  \end{equation}
  over $Quot_{d}^{\circ }\times S$, where $\mathcal{V}_{d}$ and $\mathcal{W}_{d}$ are locally free
sheaves.
\end{lemma}
\begin{proof}
  For a sufficient large integer $r$, let
  $\mathcal{W}_{d}=\pi^{*}\pi_{*}\mathcal{I}_{d}(r)\otimes \mathcal{O}(-r)$, where $\pi
  $ is the projection map from $Quot_{d}^{\circ }\times S$ to $Quot_{d}^{\circ }$. There is a
  surjective map $\mathcal{W}_{d}\to \mathcal{I}_{d}$. Let $\mathcal{V}_{d}$ be its kernel and $\mathcal{V}_{d}$
  is also locally free by Lemma 2.1.7 of \cite{huybrechts2010geometry}.
\end{proof}
Now we generalize the Quot scheme to the case of  a sequence of
non-decreasing integers $d_{\bullet}=(d_{0},d_{1},\ldots,d_{l})$, such
that $d_{0}=0$ and $d_{l}=d$. Fix a flag of vector spaces $F=\{k^{d_{1}}\subset \ldots \subset k^{d_{l}}\}$. Let
$$ Quot_{d_{\bullet}}^{\circ }=\prod_{i=1}^{k}Quot_{d_{i}-d_{i-1}}^{\circ }.$$

\begin{definition}
  For any closed point $\{\phi:k^{d}\otimes
\mathcal{O}_{S}\twoheadrightarrow \mathcal{E}_{d}\}$ of
$Quot_{d}^{\circ }$, let $\mathcal{E}_{d_i}$ be the image of
$\phi_{i}=\phi|_{k^{d_{i}}\otimes \mathcal{O}_{S}}$. By
\cite{minets18:cohom_hall_higgs} $Flag_{d_{\bullet}}^{\circ }$, the
subset of $Quot_{d}^{\circ }$ which consists of quotients of $\mathcal{E}_{d}$ such that for any $i$
$$H^{0}(\phi_{i}):k^{d_{i}}\cong H^{0}(\mathcal{E}_{d_{i}})$$
is a closed subscheme and denoted by $Flag_{d_{\bullet}}^{\circ}$. The
inclusion map is denoted by $i_{d_{\bullet}}$:
\begin{equation}
  \label{eq:def1}
  i_{d_{\bullet}}:Flag_{d_{\bullet}}^{\circ} \to Quot_{d}^{\circ }.
\end{equation}

\end{definition}

For each $i$, there is a universal quotients of coherent sheaves
$$\phi_{i}:\mathcal{O}^{d_{i}}\to \mathcal{E}_{d_{i}}$$
 over $Flag_{d_{\bullet}}\times S$. Moreover, fixing an isomorphism
$k^{d_{i}-d_{i-1}}=k^{d_{i}}/k^{d_{i-1}}$, let
$\mathcal{E}_{d_{i,i-1}}=\mathcal{E}_{d_{i}}/\mathcal{E}_{d_{i,i-1}}$. Then
$$\phi_{i,i-1}:\mathcal{O}^{d_{i}-d_{i-1}}\to
\mathcal{E}_{d_{i,i-1}}$$ is also surjective. It induces a morphism
\begin{equation}
  \label{eq:def2}
  p_{d_{\bullet}}:Flag_{d_{\bullet}}^{\circ }\to Quot_{d_{\bullet}}^{\circ }.
\end{equation}

\subsection{Two term complexes $W_{m,n}$ and $V_{n,m}$}\label{sec:res}
Given two non-negative integers $n,m$, let $d_{\bullet}=(0,m,m+n)$ and
\begin{center}
  \begin{tabular}{ccc}
  $Quot_{d_{\bullet}}^{\circ}$ & denote & $Quot_{n,m}^{\circ}$, \\
    $Flag_{d_{\bullet}}^{\circ}$ & denote & $Flag_{n,m}^{\circ}$, \\
    $p_{d_{\bullet}}$ & denote & $p_{n,m}$.
\end{tabular}
\end{center}
\begin{definition}
  Given a resolution of locally free sheaves of
$\mathcal{I}_{m}$ over $Quot_{m,n}^{\circ }\times S$ by \cref{eq:2.1},
we define the two term complexes
\begin{align*}
  W_{n,m}=\pi_{*}\mathcal{H}om(\mathcal{W}_{m},\mathcal{E}_{n})=\pi_{*}(\mathcal{W}_{m}^{\vee}\otimes \mathcal{E}_{n}) \\
  V_{n,m,}=\pi_{*}\mathcal{H}om(\mathcal{V}_{m},\mathcal{E}_{n})=\pi_{*}(\mathcal{V}_{m}^{\vee}\otimes \mathcal{E}_{n}).
\end{align*}
where $\pi$ is the projection map from $Quot_{n,m}^{\circ }\times S$
to $Quot_{n,m}^{\circ }$.

The morphism $v_{m}:\mathcal{W}_{m}\to \mathcal{V}_{m}$ naturally induces
a morphism between the two term complexes $W_{n,m}\to V_{n,m}$, which is denoted by $\psi_{n,m}$.

\end{definition}

\begin{proposition}
  \label{prop:2.6}
  $W_{n,m}$ and $V_{n,m}$ are locally free sheaves over
  $Quot_{n,m}^{\circ}$ and
  $$R^{i}\pi_{*}(\mathcal{H}om(\mathcal{W}_{m},\mathcal{E}_{n}))=0 ,\ R^{i}\pi_{*}(\mathcal{H}om(\mathcal{V}_{m},\mathcal{E}_{n}))=0$$
  for $i>0$. 
\end{proposition}

We first recall a base change theorem for cohomology of
coherent sheaves:
\begin{theorem}[Chapter \RN{3}, Theorem 12.11 of
  \cite{hartshorne2013algebraic}]
  \label{thm:2.7}
  Let $f:X\to Y$ be a projective morphism of noetherian schemes, and let
  $\mathcal{F}$ be a coherent sheaf on $X$, flat over $Y$. Let $y$ be a closed point of
  $Y$. Then for any non-negative integer, there is a natural map
  $$\phi^{i}(y):R^{i}f^{*}(\mathcal{F})\otimes k(y)\to
  H^{i}(X_{y},\mathcal{F}_{y}).$$
  \begin{enumerate}
  \item\label{2.7.1}  If $\phi^{i}(y)$ is surjective, then it is an isomorphism.
  \item\label{2.7.2}  If $R^{i}f_{*}(\mathcal{F})$ is locally free in a neighborhood of $y$, then $\phi^{i-1}(y)$
  is also surjective.
  \end{enumerate}
\end{theorem}

Now we prove the Proposition \ref{prop:2.6}.
\begin{proof}
  Let $r$ be the rank of $\mathcal{W}_{m}$ and we consider the projection
  map
  $$\pi:Quot_{n,m}^{\circ}\times S \to Quot_{n,m}^{\circ}.$$
  For any closed point $b\in Quot_{n,m}^{\circ}$,
  $\pi^{-1}(b)=b\times S$. Hence $\mathcal{W}_{m}^{\vee}\otimes
  \mathcal{E}_{n}|_{b\times S}$ has dimension $0$ and  length $nr$.
  Thus for any closed point $b\in B$
  \begin{align*}
    h^{i}(b\times S,\mathcal{W}_{m}^{\vee}\otimes
    \mathcal{E}_{n}|_{b\times S})=0, & \textrm{ if } i>0 \\
   h^{0}(b\times S,\mathcal{W}_{m}^{\vee}\otimes
    \mathcal{E}_{n}|_{b\times S})=nr, & \textrm{ if } i=0
  \end{align*}

  By Theorem \ref{thm:2.7}, we have
  $$R^{i}\pi_{*}(\mathcal{H}om(\mathcal{W},\mathcal{E}_{n}))=0$$
  for $i>0$. Let us take $i=1$, $\phi^{0}(b)$ is also
  an isomorphism by \eqref{2.7.2} of Theorem \ref{thm:2.7}. Hence
  $\phi^{0}$ is also an isomorphism and therefore $W_{n,m}$ is locally free.

  The analogous proof also holds for $V_{n,m}$.
\end{proof}

\subsection{Two Term Complexes and Flag Schemes}
\label{sec:dia}
There is a commutative diagram of coherent sheaves over $Flag_{n,m}^{\circ }\times S$, which are flat over $Flag_{n,m}^{\circ }$
\begin{equation*}
  \label{eq:fl1}
  \begin{tikzcd}
    & & & 0 \ar{d} &\\
    & & & \mathcal{V}_{m} \ar{d}{v_{m}} &\\
    & & & \mathcal{W}_{m} \ar{d}{w_{m}} &\\
   0\ar{r} & k^{n}\otimes \mathcal{O} \ar{d}{\phi_{n}} \ar{r} & k^{n+m}\otimes \mathcal{O} \ar{d}{\phi_{n+m}} \ar[r,shift right, swap]  & k^{m}\otimes \mathcal{O} \ar{d}{\phi_{m}} \ar{r} \ar[l, shift right, swap, "l"] & 0 \\
   0\ar{r} & \mathcal{E}_{n} \ar{r}{h_{n}} & \mathcal{E}_{n+m} \ar{r}{h_{n+m}} & \mathcal{E}_{m} \ar{r} \ar{d} & 0\\
   & & & 0, &
 \end{tikzcd}
\end{equation*}
where all the columns and rows are exact, and $l:k^{m}\otimes
\mathcal{O}\to k^{n+m}\otimes \mathcal{O}$ is defined by
$l(v)=(0,v)$. Then $\phi_{n+m}\circ l\circ w_{m}:\mathcal{W}_{m}\to \mathcal{E}_{n+m}$
has image $0$ in $\mathcal{E}_{m}$. So there is a unique morphism
$$t_{n,m}:\mathcal{W}_{m}\to \mathcal{E}_{n}$$
such that $h_{n}\circ
t_{n,m}=\phi_{n+m}\circ l\circ w_{m}$. $t_{n,m}$ corresponds to a
global section of $W_{n,m}$. Moreover, the diagram
 \begin{equation}
    \label{eq:2.4}
    \begin{tikzcd}
    Flag_{n,m}^{\circ } \ar{r}{t_{n,m}} \ar{d}{p_{n,m}} & W_{n,m} \ar{d}{\psi_{n,m}} \\
    Quot_{n,m}^{\circ } \ar{r}{i_{V}} & V_{n,m}    
  \end{tikzcd}
\end{equation}
is commutative, where $i_{V}$ is the zero section.

\begin{proposition}
  \label{prop:2.14}
The commutative diagram \eqref{eq:2.4} is Cartesian.
\end{proposition}

\begin{proof}
  \begin{description}
  \item[Step 1]   Let $Flag=W_{n,m}\times _{Quot_{n,m}^{\circ }}V_{n,m}$. Diagram
  \cref{eq:2.4} induces a morphism  $\tau:Flag_{n,m}^{\circ  }\to
  Flag$ and a natural transformation of functors
  $$\tau:\operatorname{Hom}(-,Flag_{n,m}^{\circ})\to
  \operatorname{Hom}(-,Flag).$$
  By Yoneda lemma, we only need to  construct another natural transformation
  $$\tau':\operatorname{Hom}(-,Flag)\to
  \operatorname{Hom}(-,Flag_{n,m}^{\circ})$$
  which is inverse to $\tau$.
\item[Step 2]
  Given a test scheme $X$ and a morphism $f:X\to Flag$, there is a
  commutative diagram:
  \begin{equation}
    \label{eq:correct1}
    \begin{tikzcd}
      X \ar{r}{X_{W}} \ar{d}{X_{Q}} & W_{n,m} \ar{d}{\psi_{n,m}} \\
      Quot_{n,m}^{\circ} \ar{r}{i_{V}}  & V_{n,m}
    \end{tikzcd}
  \end{equation}
  $X_{Q}$ induces a long exact sequence of coherent sheaves
  $$0\to
  \mathcal{V}_{m}\xrightarrow{v_{m}}\mathcal{W}_{m}\xrightarrow{w_{m}}k^{m}\otimes
  \mathcal{O}\to \mathcal{E}_{m}\to 0$$
  over $X\times S$ and a universal coherent sheaf
  $\mathcal{E}_{n}\in Coh(X\times S)$. $X_{W}$
  induces a morphism $t_{n,m}:\mathcal{W}_{m}\to \mathcal{E}_{n}$ such
  that $t_{n,m}\circ v_{m}=0$.
\item[Step 3] Consider
  \begin{align*}
     \mathfrak{t}: \mathcal{W}_{m}\to & \mathcal{E}_{n}\oplus (k^{m}\otimes \mathcal{O}) \\
    a\to &(t_{n,m}(a),w_{m}(a))
  \end{align*}
    and let $\mathcal{E}_{n+m}$ be the cokernel of $\mathfrak{t}$. The
  following sequence is exact:
  \begin{equation}
    \label{eq:correct2}
    0\to \mathcal{V}_{m}\xrightarrow{v_{m}}\mathcal{W}_{m}\xrightarrow{\mathfrak{t}} \mathcal{E}_{n}\oplus (k^{m}\otimes \mathcal{O}) \xrightarrow{p}
   \mathcal{E}_{n+m}\to  0.
 \end{equation}
 Since $\mathcal{I}_{m}=\mathcal{W}_{m}/\mathcal{V}_{m}$, equation \cref{eq:correct2} induces another exact sequence:
  $$0\to
  \mathcal{I}_{m}\xrightarrow{\mathfrak{t}'}\mathcal{E}_{n}\oplus
  (k^{m}\otimes
  \mathcal{O})\xrightarrow{p}\mathcal{E}_{n+m}\to 0.$$
  Let $p=\varepsilon_{n}\oplus p_{2}$ for
  \begin{align*}
     \varepsilon_{n}:\mathcal{E}_{n}\to \mathcal{E}_{n+m}, \quad p_{2}:k^{m}\otimes \mathcal{O}\to  \mathcal{E}_{n+m}.
  \end{align*}
  Then we have the following diagram of coherent
  sheaves over $X\times S$.
  \begin{equation}
    \label{eq:correct3}
    \begin{tikzcd}
               & 0 \ar{d}        &  0 \ar{d}          & 0 \ar{d} \\
      0 \ar{r} & 0 \ar{r} \ar{d} &  \mathcal{E}_{n}\ar{r}{id} \ar{d}{(id,0)} & \mathcal{E}_{n} \ar{r} \ar{d}{\varepsilon_{n}} & 0 \\
      0 \ar{r} & \mathcal{I}_{m} \ar{r}{\mathfrak{t}'} \ar{d} & \mathcal{E}_{n}\oplus (k^{m}\otimes \mathcal{O})\ar{r}{p} \ar{d}{(0,id)} & \mathcal{E}_{n+m} \ar{r} \ar[d,dashed,"\varepsilon_{m}"]& 0 \\
      0 \ar{r}  & \mathcal{I}_{m} \ar{r}\ar{d} & k^{m}\otimes \mathcal{O} \ar{r}{\phi_{n}} \ar{d}{id} & \mathcal{E}_{m} \ar{r} \ar{d} & 0 \\
      & 0 & 0 & 0
    \end{tikzcd}
  \end{equation}
  where all the rows and columns, except the dashed one, are exact sequences. Then
  \begin{enumerate}
  \item $\phi_{n}\circ (0,id) \circ \mathfrak{t}'=0:\mathcal{I}_{m}\to
    \mathcal{E}_{m}$, hence there is a dashed morphism
    $\varepsilon_{m}$ in the above diagram \eqref{eq:correct3} to make
    all the rows and columns to be exact.
  \item $\mathcal{E}_{n+m}$ is also flat over $X$, by the fact $\mathcal{E}_{m}$ and
    $\mathcal{E}_{n}$ are flat over $X$.
  \item $\phi_{n}\oplus id:\mathcal{O}^{n}\oplus \mathcal{O}^{m}\to \mathcal{E}^{n}\oplus \mathcal{O}_{m}$ induced a map
    $\phi_{n+m}=p\circ(\phi_{n}\oplus id):k^{n+m}\otimes \mathcal{O}\to \mathcal{E}_{n+m}$, and we have the
    following diagram:
    \begin{equation*}
      \begin{tikzcd}
         0\ar{r} & k^{n}\otimes \mathcal{O} \ar{d}{\phi_{n}} \ar{r} & k^{n+m}\otimes \mathcal{O} \ar{d}{\phi_{n+m}} \ar{r} & k^{m}\otimes \mathcal{O} \ar{d}{\phi_{m}} \ar{r}  & 0 \\
   0\ar{r} & \mathcal{E}_{n} \ar{r} \ar{d} & \mathcal{E}_{n+m} \ar{r} \ar{d} & \mathcal{E}_{m} \ar{r} \ar{d} & 0\\
   & 0 & 0 & 0, &
      \end{tikzcd}
    \end{equation*}
  \end{enumerate}
  where all the columns and rows are exact. It induces a
  morphism $\tau'(f)\in \operatorname{Hom}(X, Flag_{n,m}^{\circ })$. The above process
  actually defines a natural transformation: 
  $$\tau':\operatorname{Hom}(-,Flag)\to\operatorname{Hom}(-,Flag_{n,m}^{\circ }) $$
  which is inverse to $\tau$ by checking the correspondence morphism functorially.
  \end{description}
\end{proof}

\subsection{Refined Gysin Map} The Cartesian diagram \eqref{eq:2.4}, induces the refined Gysin map:
$$\psi_{n,m}^{!}:K(Quot_{n,m}^{\circ })\to K(Flag_{n,m}^{\circ }).$$
Moreover, for any Cartesian diagram
\begin{equation*}
  \begin{tikzcd}
    X \ar{r} \ar{d} & Flag_{n,m}^{\circ } \ar{d} \\
    Y \ar{r}        & Quot_{n,m}^{\circ }
  \end{tikzcd}
\end{equation*}
the refined Gysin map:
$$\psi_{n,m}^{!}:K(Y)\to K(X).$$
is always well-defined.
\begin{proposition}
  \label{prop:ind}
  $\psi_{n,m}^{!}$ does not depend on the choice of resolutions \cref{eq:2.1} of the
  kernel sheaf $\mathcal{I}_{m}$.
\end{proposition}

\begin{proof}
  Let
  $$0\to\mathcal{V}_{m}\xrightarrow{v}\mathcal{W}_{m}\xrightarrow{w}\mathcal{I}_{m}\to 0$$
  and
  $$0\to\mathcal{V}_{m}'\xrightarrow{v'}\mathcal{W}_{m}'\xrightarrow{w'}\mathcal{I}_{m}\to 0$$
  be two resolutions of $\mathcal{I}_{m}$. Let
  $\mathcal{W}_{m}''=\mathcal{W}_{m}\oplus
  \mathcal{W}_{m}'$. $\mathcal{W}_{m}''\xrightarrow{w\oplus
    w'}\mathcal{I}_{m}$ is surjective and let $\mathcal{V}_{m}''$ be
  its kernel. Then we
  have the following diagram of exact sequences:
  \begin{equation}
    \label{eq:correct1}
    \begin{tikzcd}
      & 0\ar{d} & 0\ar{d} & 0\ar{d} & \\
      0 \ar{r} & \mathcal{V}_{m} \ar{r}{v} \ar{d}{r} & \mathcal{W}_{m} \ar{r}{w} \ar{d}{(id,0)} & \mathcal{I}_{m} \ar{r}\ar{d} & 0 \\
      0 \ar{r} & \mathcal{V}_{m}''\ar{r}{v''} \ar{d} & \mathcal{W}_{m}'' \ar{r}{w\oplus w'}\ar{d}{0\oplus id} & \mathcal{I}_{m} \ar{r} \ar{d} & 0 \\
      0 \ar{r} & \mathcal{W}_{m}' \ar{r} \ar{d} & \mathcal{W}_{m}' \ar{r}\ar{d} & 0 \ar{r} \ar{d} & 0 \\
      & 0 & 0 & 0 & 
    \end{tikzcd}
  \end{equation}
  where $r$ is the restriction of $(id,0)$ to $\mathcal{V}_{m}$. Consider the following short exact sequence:
  \begin{equation*}
    0\to \mathcal{V}_{m}\xrightarrow{(v,r)}\mathcal{W}_{m}\oplus \mathcal{V}_{m}''\xrightarrow{(id,0)\oplus -v''} \mathcal{W}_{m}''\to 0 
\end{equation*}
and the Cartesian diagram:
\begin{equation*}
  \begin{tikzcd}
    W_{m,n}'' \ar{r}{\psi_{n,m}''} \ar{d} & V_{m,n}'' \ar{d} \\
    W_{m,n} \ar{r}{\psi_{n,m}} & V_{m,n}
  \end{tikzcd}
\end{equation*}
Moreover, by diagram \cref{eq:correct1}, we have the short exact
sequence:
$$\mathcal{V}_{m} \to \mathcal{V}_{m}''\to \mathcal{W}_{m}'\to 0,$$
which induces the exact sequence:
$$0\to W_{m,n}'\to V_{m,n}''\to V_{m,n}\to
0.$$ So the above map from $V_{m,n}''$ to $V_{m,n}$ is smooth and
$\psi_{m,n}''^{!}=\psi_{m,n}^{!}$ by Lemma \ref{lem:flat}. Similarly $\psi_{m,n}''^{!}=\psi_{m,n}'^{!}$.
\end{proof}

\subsection{An Associativity Formula }

Now let $n,m,l$ be three non-negative integers and
$d_{\bullet}=(0,n,n+m,n+m+l)$. Let
\begin{center}
  \begin{tabular}{ccc}
  $Flag_{d_{\bullet}}^{\circ }$ & denoted by & $Flag_{n,m,l}^{\circ }$,\\
  $Quot_{d_{\bullet}}^{\circ }$ & denoted by & $Quot_{n,m,l}^{\circ }$.
\end{tabular}
\end{center}
We have the following Cartesian
diagrams:

\begin{minipage}[t]{0.5\linewidth}
  \begin{equation}
    \label{correct:3}
  \begin{tikzcd}
    Flag_{n,m,l}^{\circ }\ar{r} \ar{d} & Flag^{\circ }_{n,m+l} \ar{d}{p_{n,m+l}} \\
    Quot_{n}^{\circ }\times Flag_{m,l}^{\circ } \ar{r} & Quot_{n,m+l}^{\circ }
  \end{tikzcd}
\end{equation}
\end{minipage}
\begin{minipage}[t]{0.5\linewidth}
  \begin{equation}
    \label{eq:correct4}
  \begin{tikzcd}
    Flag_{n,m,l}^{\circ }\ar{r} \ar{d} & Flag^{\circ }_{n+m,l} \ar{d}{p_{n+m,l}} \\
    Flag_{n,m}^{\circ }\times Quot_{l}^{\circ } \ar{r} & Quot_{n+m,l}^{\circ }
  \end{tikzcd}
\end{equation}
\end{minipage}
Hence we have refined Gysin maps
$$\psi_{n,m+l}^{!}:K(Quot_{n}^{\circ }\times Flag_{m,l}^{\circ })\to K(Flag_{n,m,l}^{\circ })$$
$$\psi_{n+m,l}^{!}:K(Flag_{n,m}^{\circ }\times Quot_{l}^{\circ })\to
K(Flag_{n,m,l}^{\circ }).$$
Now we are going to prove an associative formula:
\begin{proposition}
  \label{prop:ass}
  $\psi_{n,m+l}^{!}\psi_{m,l}^{!}=\psi_{n+m,l}^{!}\psi_{n,m}^{!}:K(Quot_{n,m,l}^{\circ})\to K(Flag_{n,m,l}^{\circ })$.
\end{proposition}

\begin{proof}
  \begin{description}
  \item[Step 1]   Given a sufficient large $r$, let
  \begin{flalign*}
    \mathcal{W}_{l}=\pi^{*}\pi_{*}(\mathcal{I}_{l}(r))\otimes \mathcal{O}(-r), \\
    \mathcal{W}_{m}=\pi^{*}\pi_{*}(\mathcal{I}_{m}(r))\otimes \mathcal{O}(-r), \\
    \mathcal{W}_{n}=\pi^{*}\pi_{*}(\mathcal{I}_{n}(r))\otimes \mathcal{O}(-r), \\
    \mathcal{W}_{m+n}=\pi^{*}\pi_{*}(\mathcal{I}_{n+m}(r))\otimes \mathcal{O}(-r), \\
    \mathcal{W}_{m+l}=\pi^{*}\pi_{*}(\mathcal{I}_{m+l}(r))\otimes \mathcal{O}(-r)
  \end{flalign*}
  for all $X\times S$ where $X$ are all the schemes in diagram
  \eqref{eq:correct3} and \eqref{eq:correct4} where those coherent
  sheaves are well defined, similar to the definition in lemma
  \ref{lem:res}. Then these coherent sheaves have surjective maps to
  $\mathcal{I}_{l},\mathcal{I}_{m},\mathcal{I}_{n},\mathcal{I}_{m+n},\mathcal{I}_{n+l}$ respectively, and let
  $\mathcal{V}_{l},\mathcal{V}_{m},\mathcal{V}_{n},\mathcal{V}_{m+n},\mathcal{V}_{n+l}$ to be their kernels
  respectively.

  Then we have the exact sequence:
  \begin{equation*}
    \begin{tikzcd}
      0 \ar{r} & \mathcal{V}_{n} \ar{r} \ar{d} & \mathcal{V}_{m+n} \ar{r} \ar{d} & \mathcal{V}_{m} \ar{d} \ar{r} & 0\\
      0 \ar{r} & \mathcal{W}_{n} \ar{r} & \mathcal{W}_{m+n} \ar{r} & \mathcal{W}_{m} \ar{r} & 0
    \end{tikzcd}
  \end{equation*}
  over $Flag_{n,m}^{\circ }\times Quot_{l}^{\circ }\times S$ and
  another exact sequence
  \begin{equation*}
    \begin{tikzcd}
      0 \ar{r} & \mathcal{V}_{m} \ar{r} \ar{d} & \mathcal{V}_{m+l} \ar{r} \ar{d} & \mathcal{V}_{l} \ar{d} \ar{r} & 0\\
      0 \ar{r} & \mathcal{W}_{m} \ar{r} & \mathcal{W}_{m+l} \ar{r} & \mathcal{W}_{l} \ar{r} & 0
    \end{tikzcd}
  \end{equation*}
  over $Quot_{n}^{\circ }\times Flag_{m,l}^{\circ }\times S$. Similar to Section
  \ref{sec:res}, there are locally free sheaves
  $V_{n,l},V_{m,l},V_{n,m},W_{n,l},W_{m,l},W_{n,m}$ over
  $Quot_{n,m,l}^{\circ }$, $V_{n+m,l},W_{n+m,l}$ over
  $Flag_{n,m}^{\circ }\times Quot_{l}^{\circ }$ and
  $V_{n,m+l},W_{n,m+l}$ over $Quot_{n}^{\circ }\times Flag_{m,l}^{\circ }$.
  \item[Step 2] Let $Corr_{n,m,l}^{\circ}$ be defined by the following Cartesian
  diagram:
  \begin{equation*}
    \begin{tikzcd}
      Corr_{n,m,l}^{\circ }\ar{r} \ar{d} & Flag_{n,m}^{\circ }\times Quot_{l}^{\circ } \ar{d}{p_{n,m}} \\
      Quot_{n}^{\circ }\times Flag_{m,l}^{\circ } \ar{r}{p_{m,l}} & Quot_{n,m,l}^{\circ }    
    \end{tikzcd}
  \end{equation*}
  Then the locally free sheaves $W_{n+m,m+l}$ and $V_{n+m,m+l}$ are
  also well defined on $Corr_{n,m,l}^{\circ}$. Let $\psi_{n,m}$ denote
  the morphism from $Corr_{n,m,l}^{\circ}$ to $Quot_{n}^{\circ}\times
  Flag_{m,l}^{\circ}$ and $\psi_{m,l}$ denote the morphism from
  $Corr_{n,m,l}^{\circ}$ to $Flag_{n,m}^{\circ}\times Quot_{l}^{\circ}$ by abusing the notation. By Lemma \ref{lem:kcom}, we have
  \begin{equation}
  \label{eq:2.11}
  \psi_{n,m}^{!}\psi_{m,l}^{!}=\psi_{m,l}^{!}\psi_{n,m}^{!}:K(Quot_{n,m,l}^{\circ })\to K(Corr_{n,m,l}^{\circ }).  
\end{equation}
\item[Step 3] Consider the following exact sequences of locally free sheaves over
$Quot_{n}^{\circ }\times Flag_{m,l}^{\circ }$:

\begin{equation*}
  \begin{tikzcd}
    0\ar{r} & W_{n,l} \ar{r} \ar{d} & W_{n,l+m} \ar{r}{f} \ar{d}{\psi_{n,l+m}} & W_{n,m} \ar{d}{\psi_{n,m}} \ar{r} & 0 \\
    0\ar{r} & V_{n,l} \ar{r}  & V_{n,l+m} \ar{r} & V_{n,m} \ar{r} & 0
  \end{tikzcd}
\end{equation*}
where $Corr_{n,m,l}^{\circ }$ is the pre-image of the zero section of
$\psi_{n,l+m}$ and $Flag_{n,m,l}^{\circ }$ is the pre-image of the
zero section of $\psi_{n,l+m}^{-1}$. Let $Y_{1}=f^{-1}(Corr_{n,m,l}^{\circ })$. Then by \cref{di:final1}, there
is a Cartesian diagram:

\begin{equation*}
  \begin{tikzcd}
    Flag_{n,m,l}^{\circ }\ar{r}\ar{d} & Y_{1} \ar{d}{\psi_{1}} \\
    Corr_{n,m,l}^{\circ }\ar{r} & V_{n,l}\times _{Quot_{n,m,l}^{\circ }}Corr_{n,m,l}^{\circ }.
  \end{tikzcd}
\end{equation*}
By Lemma \ref{lem:kass},
\begin{equation}
  \label{eq:correct5}
  \psi_{n,l+m}^{!}\circ \psi_{m,l}^{!}=\psi_{1}^{!}\circ \psi_{m,l}^{!}\circ \psi_{n,m}^{!}.
\end{equation}

Similarly, there is another commutative diagram of exact sequences of locally free sheaves over
$Flag_{n,m}^{\circ }\times Quot_{l}^{\circ }$:
\begin{equation*}
  \begin{tikzcd}
    0\ar{r} & W_{n,l} \ar{r} \ar{d}{} & W_{n+m,l} \ar{r}{f'} \ar{d}{\psi_{n+m,l}} & W_{m,l} \ar{d}{\psi_{m,l}} \ar{r} & 0 \\
    0\ar{r} & V_{n,l} \ar{r}  & V_{n+m,l} \ar{r} & V_{m,l} \ar{r} & 0
  \end{tikzcd}
\end{equation*}
where $Corr_{n,m,l}^{\circ }$ is the pre-image of the zero section of
$\psi_{n+m,l}$ and $Flag_{n,m,l}^{\circ }$ is the pre-image of the
zero section of $\psi_{n+m,l}^{-1}(0)$. Let
$Y_{2}=f'^{-1}(Corr_{n,m,l}^{\circ })$. There is the following Cartesian
diagram:
\begin{equation*}
  \begin{tikzcd}
    Flag_{n,m,l}^{\circ }\ar{r}\ar{d} & Y_{2} \ar{d}{\psi_{2}} \\
    Corr_{n,m,l}^{\circ }\ar{r} & V_{n,l}\times _{Quot_{n,m,l}^{\circ }}Corr_{n,m,l}^{\circ }.
  \end{tikzcd}
\end{equation*}
Still by Lemma \ref{lem:kass}
\begin{equation}
  \label{eq:correct6}
  \psi_{n+m,l}^{!}\circ \psi_{n,m}^{!}=\psi_{2}^{!}\circ \psi_{n,m}^{!}\circ \psi_{m,l}^{!}.
\end{equation}
\item[Step 4] By equation \eqref{eq:2.11}, \eqref{eq:correct5} and
  \eqref{eq:correct6}, we only need to prove that $Y_{1}$ and $Y_{2}$ are isomorphic, and
  $\psi_{1}=\psi_{2}$.

  Let $W_{n,m}\times_{Quot_{n,m,l}^{\circ}}Corr_{n,m,l}^{\circ}$
  denoted still by $W_{n,m}$ by abusing the notation for the pullback
  of vector bundles. Similarly we abuse the notation for
  $W_{n,l},W_{m,l},W_{n,l+m}$ and $W_{n+m,l}$. $\psi_{n,m},
  \psi_{m,l}$ induced global sections:
  \begin{align*}
   t_{n,m}:Corr_{n,m,l}^{\circ }\to W_{n,m},\quad  & t_{m,l}:Corr_{n,m,l}^{\circ }\to W_{m,l}
  \end{align*}
  $Y_{1},Y_{2}$ can be represented by the following Cartesian
  diagrams:
  
  \begin{minipage}{0.5\linewidth}
    \begin{equation}
    \label{tikz:correct1}
    \begin{tikzcd}
      Y_{1} \ar{r} \ar{d} & Corr_{n,m,l}^{\circ} \ar{d}{t_{n,m}} \\
      W_{n,l+m} \ar{r} & W_{n,m}
    \end{tikzcd}
  \end{equation}
  \end{minipage}
  \begin{minipage}{0.5\linewidth}
    \begin{equation}
      \label{tikz:correct2}
      \begin{tikzcd}
      Y_{2} \ar{r} \ar{d} & Corr_{n,m,l}^{\circ} \ar{d}{t_{m,l}} \\
      W_{n+m,l} \ar{r} & W_{m,l}
    \end{tikzcd}
  \end{equation}
  \end{minipage}
\item[Step 5] Recall that over $Corr_{n,m,l}^{\circ }\times S$, we have
    following diagrams:
    \begin{equation}
      \label{di:final2}
  \begin{tikzcd}
    0\ar{d} & 0\ar{d} & 0 \ar{d} & 0 \ar{d} & 0\ar{d}\\
  \mathcal{V}_{n}\ar{d} \ar[r,dotted] & \mathcal{V}_{n+m}\ar[r,dotted] \ar{d} & \mathcal{V}_{m} \ar{d}  \ar[r,dashed] &\mathcal{V}_{n+m} \ar[r,dashed] \ar[d] & \mathcal{V}_{l} \ar[d]\\
  \mathcal{W}_{n} \ar[d,"w_{n}"] \ar[r,dotted,"g_{n}"] & \mathcal{W}_{n+m}\ar[r,dotted,"g_{n+m}"] \ar[d,"w_{n+m}"] & \mathcal{W}_{m} \ar{d}{w_{m}}  \ar[r,dashed,"g_{m}"]& \mathcal{W}_{m+l} \ar[d,"w_{m+l}"] \ar[r,dashed,"g_{m+l}"] & \mathcal{W}_{l} \ar[d,"w_{l}"]\\
    k^{n}\otimes \mathcal{O} \ar{d}{\phi_{n}} \ar[r,dotted] & k^{n+m}\otimes \mathcal{O} \ar{d}{\phi_{n+m}} \ar[r,shift right, swap,dotted]  & k^{m}\otimes \mathcal{O} \ar{d}{\phi_{m}} \ar[l, shift right, swap, "l_{1}"] \ar[r,dashed] & k^{m+l}\otimes \mathcal{O} \ar[r,shift right, swap,dashed] \ar[d,"\phi_{m+l}"] & k^{l} \otimes \mathcal{O} \ar[l, shift right, swap, "l_{2}"] \ar[d,"\phi_{l}"]  \\
    \mathcal{E}_{n} \ar[r,dotted,"h_{n}"] \ar{d} & \mathcal{E}_{n+m} \ar[r,dotted,"h_{n+m}"] \ar{d} & \mathcal{E}_{m} \ar{d} \ar[r,dashed,"h_{m}"] & \mathcal{E}_{m+l} \ar[r,dashed,"h_{m+l}"] \ar{d} & \mathcal{E}_{l} \ar{d} \\
    0 & 0 & 0 & 0 & 0
  \end{tikzcd}
\end{equation}
  
where all the columns, dashed rows and dotted rows are
exact. Let $g=g_{m} \circ g_{n+m}$ and $h=h_{m}\circ h_{n+m}$, then
$$g:W_{m+l,n+m}\to W_{n+m,n+m},\quad h:W_{n+m,m+l}\to W_{m+l,m+l}$$
which are induced by the composition has kernel $W_{n,m+l}$ and
$W_{n+m,l}$ respectively. Let $\overline{t_{n,m}}=h_{n}\circ t_{n,m} \circ
g_{n+m}$ and $\overline{t_{m,l}}=h_{m}\circ t_{m,l} \circ
g_{m+l}$. A closed point $x\in Y_{1}$ corresponds to a morphism
$$x\in\operatorname{Hom}(\mathcal{W}_{m+l},\mathcal{E}_{n})$$
with the condition $ x\circ g_{m}=t_{n,m}$, which is equivalent to $x\circ
g=\overline{t_{n,m}}$.  A closed point $x\in Y_{2}$ corresponds to a
morphism $$x\in\operatorname{Hom}(\mathcal{W}_{l},\mathcal{E}_{n+m})$$
with the condition $ h_{n+m}\circ x=t_{m,l}$, which is equivalent to
$h\circ x=\overline{t_{m,l}}$.

By diagram \eqref{tikz:correct1} and \eqref{tikz:correct2},  $Y_{1}$
and $Y_{2}$ can be represented by the following
Cartesian diagrams:

\begin{equation*}
  \begin{tikzcd}
    Y_{1} \ar{r} \ar{d} & Corr_{n,m,l}^{\circ} \ar{d}{(\overline{t_{n,m}},0)} \\
    W_{n+m,m+l} \ar{r}{(g,h)} & W_{n+m,n+m}\times W_{m+l,m+l} 
  \end{tikzcd}
\end{equation*}

\begin{equation*}
  \begin{tikzcd}
    Y_{2} \ar{r} \ar{d} & Corr_{n,m,l}^{\circ} \ar{d}{(0,\overline{t_{m,l}})} \\
    W_{n+m,m+l} \ar{r}{(g,h)} & W_{n+m,n+m}\times W_{m+l,m+l} 
  \end{tikzcd}
\end{equation*}
\item[Step 6] Let $s:k^{m+l}\otimes \mathcal{O}\to k^{m}\otimes \mathcal{O}$
  defined by $s(x,y)=x$, and let $\beta=\phi_{n+m}\circ l_{1} \circ s
  \circ w_{m+l}$, where $l_{1}$ and $w_{m+l}$ were defined in \cref{di:final2}. Then $\beta \circ g=\overline{t_{n,m}}$ and $h \circ
  \beta=-\overline{t_{m,l}}$. So given $x\in W_{n+m,m+l}$ such that
  $(g,h)(x)=(\overline{t_{n,m}},0)$, we have
  $(g,h)(x-\beta)=(0,\overline{t_{m,l}})$. Hence we
construct an isomorphism between $Y_{1}$ and $Y_{2}$, and proved
$\psi_{1}^{!}=\psi_{2}^{!}$.
\end{description}

\end{proof}

\subsection{Group Actions on Quot Schemes and Flag Schemes}
In this subsection we discuss the group actions on Quot schemes and
Flag schemes. First we introduce the following notations:
\begin{enumerate}
\item Let $G_{d}=GL_{d}(k), G_{d_{\bullet}}=\prod_{i=1}^{k}G_{d_{i}-d_{i-1}}.$
$G_{d}$ has a natural action on $Quot_{d}^{\circ}$ by acting on $\mathcal{O}^{n}$ and $G_{d_{\bullet}}$ has a natural action on $Quot_{d_{\bullet}}^{\circ
}$.
\item Let $T_{d}$ be the maximal torus of $G_{d}$ formed by the
  diagonal matrices. Let $\sigma_{d}$ be the permutation group of $n$
  elements, which is the Weyl group of $G_{d}$. Let
  $T_{d_{\bullet}}=\prod_{i=1}^{k}T_{d_{i}-d_{i-1}}$ and
  $\sigma_{d_{\bullet}}=\prod_{i=1}^{k}\sigma_{d_{i}-d_{i-1}}$.
\item Let $P_{d_{\bullet}}$ be the parabolic group of $G_{d}$ which
  preserves the flag $F$. $G_{d_{\bullet}}$ is the Levi subgroup of
  $P_{d_{\bullet}}$. $P_{d_{\bullet}}$ has a natural action on
  $Flag_{d_{\bullet}}^{\circ }$.
\item $p_{d_{\bullet}}$ defined in \cref{eq:def2} and
  $i_{d_{\bullet}}$ defined in \cref{eq:def1} are $P_{d_{\bullet}}$-equivariant. 
  Let $\widetilde{Flag_{d_{\bullet}}^{\circ }}=Flag_{d_{\bullet}}^{\circ
  }\times_{P_{d_{\bullet}}}G_{d}$. $i_{d_{\bullet}}$ induces a proper
  $G_{d}$-equivariant
  morphism $$q_{d_{\bullet}}:\widetilde{Flag_{d_{\bullet}}^{\circ
    }}\to Quot_{d}^{\circ}.$$
\item We will use the notation $p_{n,m}$ for $d_{\bullet}=(0,n,n+m)$
  and $p_{n,m,l}=p_{d_{\bullet}}$ for
  $d_{\bullet}=(0,n,n+m,n+m+l)$. The same principle holds for other
  notations, like $q_{d_{\bullet}}$,$G_{d_{\bullet}}$,$T_{d_{\bullet}}$ and so on.
\end{enumerate}

\begin{remark}
  All the refined Gysin map in the previous sections can be defined
equivariantly by the same constructions, and \cref{prop:ind}
and \cref{prop:ass} also have equivariant version.
\end{remark}

Moreover, there is a $P_{n,m}$ action on the vector bundle $W_{n,m}$
over $Quot_{n,m}^{\circ}$ by the following method. \cref{eq:2.1}
induces a $G_{m,n}$-equivariant homomorphism from $Hom(\mathcal{O}^{m},\mathcal{O}^{n})$ to
$Hom(\mathcal{W}_{m},\mathcal{E}_{n})$. 

 Let $U_{n,m}$ be the kernel of the projection map from $P_{n,m}$
 to $G_{n,m}$, then $U_{n,m}$ is isomorphic to
 $Hom(\mathcal{O}^{m},\mathcal{O}^{n})$. Then
 $i(x)(y)=\gamma_{m,n}(x)+y$ defines a group action $U_{n,m}$ of
 on $W_{n,m}$. Moreover, for any $g\in G_{n,m},x\in U_{n,m},u\in W_{n,m}$,
 $(g^{-1}i(x)g)(u)=g^{-1}(g(u)+\gamma_{m,n}(x))=u+\gamma_{m,n}(x)=i(x)(u)$
 by the $G_{m,n}$ equivariance of the morphism $\gamma_{m,n}$. Then
 by the following lemma, $i$ can be extended to a $P_{n,m}$ action on
 $W_{n,m}$.

 \begin{lemma}
  Let
  $$1\to \Gamma \to P \xrightarrow{proj} G\to 1$$
  be an exact sequence of algebraic groups and let
  $$i:G\subset P$$
  a closed immersion such that $proj\circ i=id$.
  Let $a$ be a group action of $\Gamma$ on $X$, $b$ be a group action of $G$ on $X$
  such that
  $$a(g)^{-1}b(u)a(g)=b(g^{-1}ug)$$ for all $g\in G,u\in \Gamma$.
  Then for any $x\in P$, there exists a unique decomposition
  $x=gu,g\in G,u\in \Gamma$, and let
  $c(x)=b(u)a(g)$. Then $c$ is a group action of $P$ on $X$.
\end{lemma}

\subsection{Torus actions on Quot Schemes}
Let $T_{d}\subset G_{d}$ be the maximal torus consisting of diagonal
  matrices.
  \begin{lemma}Let $(Quot_{d}^{\circ})^{T_{d}}$ be the fixed locus of
    $Quot_{d}^{\circ}$ with $T_{d}$ action.
  \label{lem:2.4}
  \begin{enumerate}
  \item $(Quot_{d}^{\circ })^{T_{d}}=(Quot_{1}^{\circ })^{d}=S^{d}$.
  \item   Let $pr_{ij}:S^{d+1}\to S\times S$ be the projection to $i$-th and $j$-th
    factors. Let $\mathcal{E}_{d}$ and $\mathcal{I}_{d}$ be the
    universal sheaf and kernel sheaf over  $(Quot_{1}^{\circ })^{d}\times S=S^{d+1}$. Then
  $$\mathcal{E}_{d}=\bigoplus_{i=1}^{d}pr_{i,d+1}^{*}(\mathcal{O}_{\Delta}),\quad \mathcal{I}_{d}=\bigoplus_{i=1}^{d}pr_{i,d+1}^{*}(\mathcal{I}_{\Delta}),$$
  where $\Delta:S\to S\times S$ is the diagonal map, $\mathcal{I}_{\Delta}$ is the ideal sheaf of
  $\Delta$ and $\mathcal{O}_{\Delta}$ is the structure sheaf of diagonal.
  \end{enumerate}

\end{lemma}
\begin{proof}
  See Lemma 3.1 of \cite{minets18:cohom_hall_higgs}
\end{proof}
Next we consider the $T_{n,m}$-fixed locus of $Quot_{n,m}^{\circ }$ for
two non-negative integers $n$ and $m$.

\begin{description}[align=left, leftmargin=0pt, labelindent=\parindent, listparindent=\parindent, labelwidth=0pt, itemindent=!]
\item[Case 1] Let $m=1$ and $n=1$,  $Quot_{1}^{\circ }\times Quot_{1}^{\circ }\times S=S^{3}$ and still let $pr_{ij}:S^{3}\to S\times S$ be the projection to the $i$-th and $j$-th factor. Then we have the following projection lemma:

\begin{lemma}[Projection Lemma]
  \label{lem:proj}
  $$pr_{12*}\mathcal{H}om(pr_{13}^{*}\mathcal{W}_{1},pr_{23}^{*}\mathcal{O}_{\Delta})=\mathcal{W}_{1}^{\vee},\quad pr_{12*}\mathcal{H}om(pr_{13}^{*}\mathcal{V}_{1},pr_{23}^{*}\mathcal{O}_{\Delta})=\mathcal{V}_{1}^{\vee}.$$
\end{lemma}

\begin{proof}
  Let
  $$\Delta_{23}=id\times \Delta:S\times S\to S\times S\times S.$$
  Then $pr_{23}^{*}\mathcal{O}_{\Delta}=\mathcal{O}_{\Delta_{23}}$, and
  $pr_{13}\Delta_{23}=pr_{12}\Delta_{12}=id$. Moreover,
  \begin{align*}
    \mathcal{H}om(pr_{13}^{*}\mathcal{W}_{1},pr_{23}^{*}\mathcal{O}_{\Delta}) &=pr_{13}^{*}\mathcal{W}_{1}^{\vee}\otimes pr_{23}^{*}\mathcal{O}_{\Delta}  \\
                                          &=pr_{13}^{*}\mathcal{W}_{1}^{\vee}\otimes \Delta_{23} \\
                                          &=\Delta_{23*}(\Delta_{23}^{*}pr_{13}^{*}\mathcal{W}_{1}^{\vee})\\
    &=\Delta_{23*}\mathcal{W}_{1}^{\vee}.
  \end{align*}
  Hence
  \begin{align*}
    pr_{12*}\mathcal{H}om(pr_{13}^{*}\mathcal{W}_{1},pr_{23}^{*}\mathcal{O}_{\Delta}) &= pr_{12*}\Delta_{23*}\mathcal{W}_{1}^{\vee} \\
    &=\mathcal{W}_{1}^{\vee}
  \end{align*}
And the analogous proof also holds for $\mathcal{V}_{1}$.
\end{proof}
\item[Case 2] 
For the general case,
$$(Quot_{n}^{\circ }\times Quot_{m}^{\circ })^{ T_{n,m}}=(Quot_{1}^{\circ })^{n}\times (Quot_{1}^{\circ })^{m}=S^{n}\times S^{m}.$$
Let $\mathcal{E}_{m},\mathcal{E}_{n}|_{S^{n}\times S^{m}\times S}\in
Coh(S^{n}\times S^{m}\times S)$ be
the restrictions of universal sheaves to $S^{m}\times S^{n}\times S$, then
  $$\mathcal{E}_{m}|_{S^{n}\times S^{m}\times S}=\bigoplus_{i=n+1}^{n+m}pr_{i,n+m+1}^{*}(\mathcal{O}_{\Delta}),\quad \mathcal{E}_{n}|_{S^{n}\times S^{m}\times S}=\bigoplus_{j=1}^{n}pr_{j,n+m+1}^{*}(\mathcal{O}_{\Delta}).$$
    
Taking a sufficiently large $r$, by \cref{lem:res} there is a
surjection map from $\mathcal{W}_{m}=\pi^{*}\pi_{*}(\mathcal{I}_{m}(r))\otimes \mathcal{O}(-r)$ to $\mathcal{I}_{m}$ and let $\mathcal{V}_{m}$ be
its kernel, where $\pi$ is the first projection map form
$Quot_{n,m}^{\circ }\times S$ to $Quot_{n,m}^{\circ }$. Let
$\mathcal{W}=\pi_{1}^{*}\pi_{1*}(\mathcal{I}_{\Delta}(r))\otimes \mathcal{O}(-r)$, where $\pi_{1}$ is the projection from $S\times S$
to $S$, and it also has a surjection to $\mathcal{I}_{\Delta}$. Let
$\mathcal{V}$ denote its kernel. Then 
  $$\mathcal{W}_{m}|_{S^{n}\times S^{m}\times S}=\bigoplus_{i=n+1}^{n+m}pr_{i,n+m+1}^{*}\mathcal{W},\quad \mathcal{V}_{m}|_{S^{n}\times S^{m}\times S}=\bigoplus_{j=1}^{m}pr_{j,n+m+1}^{*}\mathcal{V}$$

By Lemma \ref{lem:proj},
\begin{align*}
  W_{n,m}|_{S^{n}\times S^{m}} &=\bigoplus_{i=1}^{n}\bigoplus_{j=n+1}^{n+m}\frac{z_{i}}{z_{j}}\pi_{*}\mathcal{H}om(pr_{j,n+m+1}^{*}\mathcal{W}_{1},pr_{i,n+m+1}^{*}(\mathcal{O}_{\Delta})) \\
          &= \bigoplus_{i=1}^{n}\bigoplus_{j=n+1}^{n+m}\frac{z_{i}}{z_{j}}pr_{ij}^{*}\mathcal{W}^{\vee}_{1}  \\
 V_{n,m}|_{S^{n}\times S^{m}}   &=\bigoplus_{i=1}^{n}\bigoplus_{j=n+1}^{n+m}\frac{z_{i}}{z_{j}}pr_{ij}^{*}\mathcal{V}_{1}^{\vee}
\end{align*}

Thus
$$W_{n,m}^{\vee}|_{S^{n}\times S^{m}}=\bigoplus_{i=1}^{n}\bigoplus_{j=n+1}^{n+m}\frac{z_{j}}{z_{i}}pr_{ij}^{*}\mathcal{W}_{1},\quad V_{n,m}^{\vee}|_{S^{n}\times S^{m}}=\bigoplus_{i=1}^{n}\bigoplus_{j=n+1}^{n+m}\frac{z_{j}}{z_{i}}pr_{ij}^{*}\mathcal{V}_{1}.$$
We have the following exact sequence
$$ 0\to V_{n,m}^{\vee}\to W_{n,m}^{\vee}\to \bigoplus_{i=1}^{n}\bigoplus_{j=n+1}^{n+m}\frac{z_{j}}{z_{i}}\mathcal{O}\to \bigoplus_{i=1}^{n}\bigoplus_{j=n+1}^{n+m}\frac{z_{j}}{z_{i}}\mathcal{O}_{\Delta ij}\to 0$$
which induces the following equation
\begin{equation}
  \label{eq:2.5}
  \frac{[\wedge^{\bullet}V_{n,m}^{\vee}]}{[\wedge^{\bullet}W_{n,m}^{\vee}]}=\prod_{i=1}^{n}\prod_{j=n+1}^{n+m}\frac{\wedge^{\bullet}[\frac{z_{j}}{z_{i}}\mathcal{O}_{\Delta ij}]}{1-\frac{z_{j}}{z_{i}}}
\end{equation}
\end{description}

\section{K-Theoretic Hall Algebra of a Surface}\label{sec:hall}

Let
$$Coh=\bigsqcup_{d=0}^{\infty }Coh_{d}$$
be the moduli stack of $0$-dimensional coherent sheaves over $S$ and
 $Coh_{d}$ be the moduli stack of dimension $0$, degree $d$ coherent
 sheaves on $S$. In this section, we will construct the K-theoretic
 Hall algebra on $K(Coh)$ and prove that it is associative.

 \subsection{Refined Gysin Map and Multiplication on $K(Coh)$}

Recall the fact that all dimension $0$ coherent sheaves are generated
by their global sections, which induces
$$Coh_{d}=[Quot_{d}^{\circ }/G_{d}]$$

Thus
\begin{equation}
  K(Coh)=\bigoplus_{d=0}^{\infty }K(Coh_{d})=\bigoplus_{d=0}^{\infty }K^{G_{d}}(Quot_{d}^{\circ })
\end{equation}

Let $\widetilde{\psi_{n,m}}$ be the composition of following
morphisms:
$$K^{G_{n,m}}(Quot_{n,m}^{\circ }) \xrightarrow{proj_{n,m}} K^{P_{n,m}}(Quot_{n,m}^{\circ })\xrightarrow{\psi_{n,m}^{!}}K^{P_{n,m}}(Flag_{n,m}^{\circ })$$
where $proj_{n,m}$ is induced by the natural projection from $P_{n,m}$
to $G_{n,m}$. Let $\widetilde{q_{n,m*}}$ be the composition of
following morphisms:
$$K^{P_{n,m}}(Flag_{n,m}^{\circ }) \xrightarrow{ind_{P_{n,m}}^{G_{n,m}}}K^{G_{n+m}}(\widetilde{Flag_{n,m}^{\circ }})\xrightarrow{q_{n,m*}} K^{G_{n+m}}(Quot_{n+m}^{\circ })$$

Now we consider the following diagram:

\begin{equation}
  \label{eq:3.1}
  \begin{tikzcd}
   &   K^{P_{n,m}}(Flag_{n,m}^{\circ }) \ar{rd}{\widetilde{q_{n,m*}}}&   \\
    K^{G_{n,m}}(Quot_{n,m}^{\circ }) \ar{ru}{\psi_{n,m}^{!}}  &  & K^{G_{n+m}}(Quot_{n+m}^{\circ }) \\
\end{tikzcd}
\end{equation}
and let
$$*^{K(Coh)}_{n,m}=\widetilde{q_{n,m*}}\circ \psi_{n,m}^{!}:K^{G_{m}}(Quot^{\circ}_{n})\otimes K^{G_{n}}(Quot^{\circ}_{m})\to K^{G_{n+m}}(Quot^{\circ }_{n+m}),$$
which induces a morphism 
$$*^{KCoh}:K(Coh)\otimes K(Coh)\to K(Coh).$$

\begin{definition}
  \label{def:hall}
  We define $(K(Coh),*^{K(Coh)})$ to be the K-theoretical Hall algebra
  associated to $S$, with unit given by $1 \in K(Coh_0) \cong \mathbb{Z}$.
\end{definition}
\begin{remark}
  The $S=\mathbb{A}^2$ version of this construction was given by Schiffmann and Vasserot in \cite{schiffmann2013} (see also  \cite{schiffmann17:hall} for quivers).
 \end{remark}

In this section, we will prove that $(K(Coh),*^{K(Coh)})$ is associative.
\begin{theorem}\label{thm:ass}
  The K-theoretic Hall algebra $(K(Coh),*^{K(Coh)})$ is associative.
\end{theorem}
\begin{proof}
  Given any three non-negative integers $n,m,l$, then
  $$*_{n+m,l}^{K(Coh)}\circ *_{n,m}^{K(Coh)}:K(Coh_{n}(S))\otimes K(Coh_{m}(S))\otimes K(Coh_{l}(S))\to K(Coh_{n+m+l}(S))$$
  is induced by the diagram:
  \begin{equation*}
    \begin{tikzcd}
      K^{P_{n,m,l}}(Flag_{n,m}^{\circ }\times Quot_{l}^{\circ }) \ar{rd}{\widetilde{q_{n,m*}}} & K^{P_{n+m,l}}(Flag_{n+m,l}^{\circ }) \ar{rd}{\widetilde{q_{n+m,l*}}}& \\
      K^{P_{n,m,l}}(Quot_{n,m,l}^{\circ }) \ar{u}{\psi_{n,m}^{!}} & K^{P_{n+m,l}}(Quot_{n+m,l}^{\circ }) \ar{u}{\psi_{n+m,l}^{!}} & K^{G_{n+m+l}}(Quot_{n+m+l}^{\circ })
    \end{tikzcd}
  \end{equation*}
  Moreover, there is the following commutative diagram by Lemma
  \ref{lem:1.1}
  \begin{equation*}
    \begin{tikzcd}
      K^{P_{n,m,l}}(Flag_{n,m,l}^{\circ }) \ar{r}{\widetilde{q_{n+m,l*}}} & K^{P_{n+m,l}}(Flag_{n+m,l}^{\circ }) \\
      K^{G_{n,m,l}}(Flag_{n,m}^{\circ }\times Quot_{l}^{\circ }) \ar{u}{\psi_{n,m+l}^{!}} \ar{r}{\widetilde{q_{n+m,l*}}} &  K^{G_{n+m,l}}(Quot_{n+m,l}^{\circ }) \ar{u}{\psi_{n+m,l}^{!}}
    \end{tikzcd}
  \end{equation*}
  So
  $$*_{n+m,l}^{K(Coh)}\circ
  *_{n,m}^{K(Coh)}=\widetilde{q_{n+m,l*}}\circ
  \widetilde{q_{n,m*}}\circ \psi_{n+m,l}^{!}\circ \psi_{n,m}^{!}$$
  $$*_{n,m+l}^{K(Coh)}\circ *_{m,l}^{K(Coh)}=\widetilde{q_{n,m+l*}}\circ \widetilde{q_{m,l*}}\circ \psi_{n,m+l}^{!}\circ \psi_{m,l}^{!}.$$ Notice
  $$\widetilde{q_{n+m,l*}}\circ
  \widetilde{q_{n,m*}}=\widetilde{q_{n,m+l*}}\circ
  \widetilde{q_{m,l*}}$$
  by the fact that $q_{n+m,l}\circ q_{n,m}=q_{n,m,l}=q_{n,m+l}\circ
  q_{m.l}$ and
  $$\psi_{n+m,l}^{!}\circ \psi_{n,m}^{!}=\psi_{n,m+l}^{!}\circ
  \psi_{m,l}^{!}$$
  by Proposition \ref{prop:ass}. So
  $$*_{n+m,l}^{K(Coh)}\circ *_{n,m}^{K(Coh)}=*_{n,m+l}^{K(Coh)}\circ *_{m,l}^{K(Coh)}$$
  and hence $*^{K(Coh)}$ is associative.
\end{proof}
\section{Shuffle Algebra and K-theoretical Hall algebra}\label{sec:shuffle}
In this section, we study the equivariant K-theory of Quot schemes
through the localization theorem \ref{thm:1.4}, and construct a
homomorphism from the K-theoretical Hall algebra to a version of shuffle algebra considered by \cite{negut2017shuffle}.

\subsection{Localization for K(Coh)}

Let $d$ be a non-negative integer, and $\sigma_{d}$ the permutation group of order $d$, which is also the Weyl group of maximal torus $T_{d}$.

By Theorem \ref{thm:1.3}, $K^{G_{d}}(Quot_{d}^{\circ
})=(K^{T_{d}}(Quot_{d}^{\circ }))^{\sigma_{d}}.$ By Lemma
\ref{lem:2.4}, $(Quot_{d})^{T_{d}}=S^{d}.$ and
$K^{T_{d}}(S^{d})_{loc}=K(S^{d})(z_{1},\ldots,z_{d})^{Sym}$, where $Sym$ means invariant under the $\sigma_{d}$ action.
Then by localization theorem \ref{thm:1.4}, there is an isomorphism
\begin{equation}
  K^{T_{d}}(S^{d})_{loc}^{Sym}\xrightarrow{i_{d*}} K^{T_{d}}(Quot_{d})^{\sigma_{d}}_{loc}.
\end{equation}
Let $l_{d}:K^{T_{d}}(Quot_{d}))^{\sigma_{d}}\to (K^{T_{d}}(Quot_{d})^{\sigma_{d}}_{loc}$ be the natural localization map, and

\begin{equation}
  \label{eq:tau}
  \tau_{d}=l_{d}\circ i_{d*}^{-1}:(K^{T_{d}}(Quot_{d}))^{\sigma_{d}}\to K(S^{d})(z_{1},\ldots,z_{d})^{Sym}
\end{equation}

 Let 
$$\tau=\bigoplus_{d=0}^{\infty }\tau_{d}:\bigoplus_{d=0}^{\infty }K^{G_{d}}(Quot_{d}^{\circ })\to \bigoplus_{d=0}^{\infty }K(S^{d})(z_{1},\ldots,z_{d})^{Sym}$$
\subsection{Shuffle Algebra Revisited} Now we recall the definition of shuffle
algebra associated to $S$, which is defined in \cite{negut2017shuffle}.

\begin{definition}
\label{def:full surface big}

Consider the abelian group:
$$
Sh= \bigoplus_{d=0}^\infty K_{S^{d}}(z_1,...,z_d)^{Sym}
$$
with the following associative product:
\begin{equation}
\label{eqn:gen shuffle product}
R(z_1,...,z_n) *^{Sh} R'(z_1,...,z_{m}) = 
\end{equation}
$$
= Sym \left[ (R \boxtimes 1^{\boxtimes m}) (z_1,...,z_n) (1^{\boxtimes n} \boxtimes R') (z_{n+1},...,z_{n+m}) \prod_{i=1}^{n}\prod_{j=n+1}^{n+m} \zeta^S_{ij} \left( \frac {z_{j}}{z_{i}} \right) \right]
$$
where:
$$
\zeta^S_{ij}(x) =  \frac{[\wedge^\bullet( x \cdot \mathcal{O}_{\Delta_{ij}})]}{(1-x)(1-\frac{1}{x})} \in K_{S^{n+m}}(x)
$$
and $\mathcal{F}\boxtimes\mathcal{G}=pr_{n}^{*}(\mathcal{F})\otimes pr_{m}^{*}(\mathcal{G})$ for $\mathcal{F}\in K(S^{n})$,
$\mathcal{G}\in K(S^{m})$. Here $pr_{n},pr_{m}$ are the respective projection map from
$S^{n}\times S^{m}$ to $S^{n}$ and $S^{m}$. $(Sh,*^{Sh})$ is defined
to be the \textbf{shuffle algebra} associated to $S$.
\end{definition}

\begin{remark}
  The definition of shuffle algebra in our paper is slightly different
 from the definition in \cite{negut2017shuffle}, but they differ by a straightforward automorphism.
\end{remark}

\begin{theorem}
  \label{thm:3.4}
  $\tau$ is an algebra homomorphism between $(K(Coh),*^{K(Coh)})$ and $(Sh,*^{Sh})$.
\end{theorem}
\begin{proof}
  \begin{description}[align=left, leftmargin=0pt, labelindent=\parindent, listparindent=\parindent, labelwidth=0pt, itemindent=!]
  \item[Step 1] Let $F=S^{n+m}\times _{Quot_{n,m}^{\circ
      }}Flag_{n,m}^{\circ }$, and consider the following Cartesian diagram:
  \begin{equation}
    \begin{tikzcd}
      F \ar{r}{i_{F}} \ar{d} & Flag_{n,m}^{\circ }\ar{r}\ar{d} & W_{n,m} \ar{d}{\psi_{n,m}} \\
      S^{n+m} \ar{r}{i} & Quot_{n,m}^{\circ }\ar{r} & V_{n,m}  \\
    \end{tikzcd}
  \end{equation}
  which defines a refined Gysin map
  $$\psi_{n,m}'^{!}:K^{T_{n+m}}(S^{n+m})\to K^{T_{n+m}}(F).$$ We prove
  that there is the commutative diagram
  \begin{equation}
    \label{eq:2.9}
      \begin{tikzcd}
      K^{T_{n+m}}(S^{n+m}) \ar{r}{j_{*}}  & K^{T_{n+m}}(F)  \ar{r}{i_{F*}}& K^{T_{n+m}}(Flag_{n,m}^{\circ }) \\
       & K^{T_{n+m}}(S^{n+m}) \ar{r}{i_{*}} \ar{u}{\psi_{n,m}'^{!}} \ar{ul}{\rho_{n,m}} & K^{T_{n+m}}(Quot_{n,m}^{\circ }) \ar{u}{\psi^{!}_{n,m}}
      \end{tikzcd}
    \end{equation}
    where
     $$\rho_{n,m}(\mathcal{F})=\prod_{i=1}^{n}\prod_{j=n+1}^{n+m}\frac{\wedge^{\bullet}[\frac{z_{j}}{z_{i}}\mathcal{O}_{\Delta ij}]}{1-\frac{z_{j}}{z_{i}}}\mathcal{F}.$$
    In fact, by \cref{lem:1.1}, $i_{F_{*}}\circ
     \psi_{n,m}'^{!}=\psi_{n,m}^{!}\circ i_{*}$. Moreover,
  $W_{n,m}^{T_{n+m}}=V_{n+m}^{T_{n+m}}=S^{n}\times S^{m}$. By
  \cref{eq:1.4} and \cref{lem:1.8}, $\psi_{n,m}'^{!}=j_{*}\circ \rho_{n,m}$.
\item[Step 2] Let
  $$\widetilde{S^{n+m}}=S^{n+m}\times_{\sigma_{n,m}}\sigma_{n+m}.$$
  The natural action of $\sigma_{n+m}$ on $S^{n+m}$ induces a projection
  map $s:\widetilde{S^{n+m}}\to S^{n+m}$. Then
  $$(\widetilde{Flag_{n,m}^{\circ })}^{T_{n+m}}=\widetilde{S^{n+m}}.$$

  There is the following commutative diagram
  \begin{equation}
    \label{eq:2.8}
    \begin{tikzcd}
      K^{T_{n+m}}(S^{n+m}\times _{\sigma_{m}\times \sigma_{n}}\sigma_{n+m}) \ar{r}{i_{*}} \ar{d}{s_{*}} & K^{T_{n+m}}(\widetilde{Flag_{n,m}^{\circ }}) \ar{d}{q_{n+m*}}\\
        K^{T_{n+m}}(S^{n+m}) \ar{r}{i_{*}} & K^{T_{n+m}}(Quot_{n+m}^{\circ })
    \end{tikzcd}
  \end{equation}
  by the fact $i_{*}\circ s_{*}=(i\circ s)_{*}=(q_{n+m}\circ
  i)_{*}=q_{n+m*}\circ i_{*}$. And by \cref{lem:1.9}, there is the
  following commutative diagram:
  \begin{equation}
    \label{eq:2.7}
    \begin{tikzcd}
      K^{T_{n+m}}(S^{n+m}\times _{\sigma_{m}\times \sigma_{n}}\sigma_{n+m}) \ar{r}{i_{*}}  & K^{T_{n+m}}(\widetilde{Flag_{n,m}^{\circ }}) \ar{d}{j_{n,m}^{*}} \\
       K^{T_{n+m}}(S^{n+m}) \ar{r}{i_{*}} \ar{u}{\hat{s}^{*}} &  K^{T_{n+m}}(Flag_{n+m}) 
    \end{tikzcd}
  \end{equation}
  where
  $\hat{s}^{*}(\mathcal{F})=s^{*}(\frac{\mathcal{F}}{\prod_{i=1}^{n}\prod_{j=n+1}^{n+m}(1-\frac{z_{i}}{z_{j}})})$
  and $j_{n,m}$ is the inclusion map from $Flag_{n,m}^{\circ }$ to
  $\widetilde{Flag_{n,m}^{\circ }}$.
\item[Step 3] By \cref{eq:2.9}, \cref{eq:2.8} and \cref{eq:2.7}, we have the following commutative diagram:
   \begin{equation*}
     \begin{tikzcd}
         K^{T_{n+m}}(S^{n+m})_{loc}^{\sigma_{n,m}}\ar{d}{id} & K^{T_{n+m}}(Quot_{n,m}^{\circ })^{\sigma_{n,m}} \ar{d}{id} \ar{l} & K^{G_{n,m}}(Quot_{n,m}^{\circ }) \ar{d} \ar{l}\\
        K^{T_{n+m}}(S^{n+m})_{loc}^{\sigma_{n,m}}\ar{d}{\rho_{n,m}} & K^{T_{n+m}}(Quot_{n,m}^{\circ })^{\sigma_{n,m}} \ar{d}{\psi_{n,m}^{!}} \ar{l}{i_{*}^{-1}} &   K^{P_{n,m}}(Quot_{n,m}^{\circ }) \ar{d}{\psi_{n,m}^{!}} \ar{l}\\
    K^{T_{n+m}}(S^{n+m})_{loc}^{\sigma_{n,m}}  \ar{d}{\hat{s}^{*}} & K^{T_{n+m}}(Flag_{n,m}^{\circ })^{\sigma_{n,m}} \ar{l}{i_{*}^{-1}} & K^{P_{n,m}}(Flag_{n,m}^{\circ }) \ar{l} \ar{d}{ind_{P}^{G}}\\
    K^{T_{n+m}}(\widetilde{S^{n+m}})_{loc}^{\sigma_{n,m}} \ar{d}{s_{*}}  & K^{T_{n+m}}(\widetilde{Flag_{n,m}^{\circ }})^{\sigma_{n+m}} \ar{d}{q_{n+m*}} \ar{u}{j_{n,m}^{*}} \ar{l}{i_{*}^{-1}} & K^{G_{n+m}}(\widetilde{Flag_{n+m}^{\circ }}) \ar{d}{q_{n+m*}} \ar{l}\\
    K^{T_{n+m}}(S^{n+m})_{loc}^{\sigma_{n+m}} & K^{T_{n+m}}(Quot_{n+m}^{\circ })^{\sigma_{n+m}} \ar{l}{i_{*}^{-1}} &  K^{G_{n+m}}(Quot_{n+m}^{\circ }) \ar{l}
    \end{tikzcd}
  \end{equation*}
  Notice $$*_{n,m}^{Sh}=s_{*}\circ \hat{s}^{*}\circ \rho_{n,m}$$ and
  $$*_{n,m}^{K(Coh)}=q_{n+m*}\circ ind_{P_{n,m}}^{G_{n+m}}\circ \psi_{n,m}^{!}.$$
  Thus we have the commutative diagram:
  \begin{equation*}
    \begin{tikzcd}
      K^{T_{n+m}}(S^{n+m})_{loc}^{\sigma_{n}\times \sigma_{m}} \ar{d}{*_{n,m}^{Sh}} & K^{G_{n+m}}(Quot_{n,m}^{\circ }) \ar{d}{*_{n,m}^{K(Coh)}} \ar{l}{\tau_{n}\otimes \tau_{m}} \\
      K^{T_{n+m}}(S^{n+m})_{loc}^{\sigma_{n+m}} & K^{G_{n+m}}(Quot_{n+m}^{\circ }) \ar{l}{\tau_{n+m}}
    \end{tikzcd}
  \end{equation*}
  i.e. $\tau$ forms a homomorphism between $*^{K(Coh)}$ and $*^{Sh}$.
\end{description}
\end{proof}

\bibliographystyle{amsalpha}
\bibliography{main.bib}
\end{document}